\documentclass[11pt,leqno]{article}
\usepackage{amsmath, amscd, amsthm, amssymb, graphics, xypic, mathrsfs, setspace, fancyhdr, times, bm, pdfsync, enumitem}
\usepackage[usenames, dvipsnames, svgnames, table]{xcolor}
\usepackage[colorlinks=true,pagebackref=true]{hyperref} 
\hypersetup{backref}

\newcommand{\trdeg}{\operatorname{tr.deg}}
\newcommand{\Spec}{\operatorname{Spec}}

\newcommand{\GW}{\mathrm{GW}}
\newcommand{\W}{\mathrm{W}}
\newcommand{\I}{\mathrm{I}}
\newcommand{\MW}{\mathrm{MW}}

\newcommand{\Z}{{\mathbb Z}}
\newcommand{\ZZ}{\Z}
\newcommand{\N}{{\mathbb N}}
\newcommand{\real}{{\mathbb R}}

\newcommand{\OO}{{\mathcal O}}
\newcommand{\A}{{\mathbb A}}
\newcommand{\PP}{{\mathbb P}}

\newcommand{\CHW}{{\widetilde{\mathrm{CH}}}}
\newcommand{\CH}{{\mathrm{CH}}}
\renewcommand{\deg}{\operatorname{deg}}

\newcommand{\KMW}{\mathrm{K}^{\MW}}
\newcommand{\KM}{\mathrm{K}^{\mathrm{M}}}
\renewcommand{\H}{{{\mathrm H}}}

\newtheorem{theorem}{Theorem}[section]
\newtheorem{lem}[theorem]{Lemma}
\newtheorem{cor}[theorem]{Corollary}
\newtheorem{prop}[theorem]{Proposition}

\theoremstyle{definition}
\newtheorem{defn}[theorem]{Definition}
\newtheorem{ex}[theorem]{Example}

\theoremstyle{remark}
\newtheorem{rem}[theorem]{Remark}

\numberwithin{equation}{section}

\begin{document}

\title{Lectures on Chow-Witt groups}

\author{Jean Fasel}


\date{}
\maketitle

\begin{abstract}
We provide a toolkit to work with Chow-Witt groups, and more generally with the homology and cohomology of the Rost-Schmid complex associated to Milnor-Witt $K$-theory. 
\end{abstract}

\tableofcontents

\section*{Introduction}

The purpose of these lectures (given in Essen in June 2018) is to give a quick introduction to Chow-Witt groups not involving prior knowledge in either Chow groups or the theory of symmetric bilinear forms. The goal is to stay as elementary as possible, avoiding altogether complicated proofs and focusing on explicit expressions allowing to perform computations. The lectures are based on the construction of the Rost-Schmid complex in \cite{Morel08} and on a primer on Milnor-Witt cohomology in preparation (jointly with Baptiste Calm\`es). The reader will eventually be able to refer to the latter for complete details, and in particular for exhaustive proofs of the results stated here.
The program of the lectures is roughly the following. In the first part, Milnor-Witt $K$-theory of a field $F$ is introduced, following the original presentation of \cite[\S 3]{Morel08}. We state a few elementary results aimed at understanding the nature of this $K$-theory, mixing up Milnor $K$-theory and symmetric bilinear forms. Then, we pass to the fundamental construction of the residue homomorphism associated to a discrete valuation field. The subtle point of the construction is that this homomorphism doesn't depend only on the chosen valuation, but also on the choice of a uniformizing parameter. This fact parallels the situation of Witt groups, and there is a well known procedure to transform it in a homomorphism independent of the choice of the parameter. This requires the introduction of Milnor-Witt $K$-theory \emph{twisted} by some extra information and leads to interesting but sometimes technically involved constructions. The major novelty here (with respect to \cite{Morel08} or \cite{Fasel08a}) is the notion of Milnor-Witt $K$-theory twisted by a \emph{graded line bundle} following the original idea of \cite{Deligne87}. Having the twisted residue homomorphisms at hand, we introduce in the second lecture the Rost-Schmid complexes, which are our main object of study. These complexes are defined for schemes essentially of finite type over a field $k$, but have some special properties when the scheme is essentially smooth. This leads to the notions of homological Rost-Schmid complex and cohomological Rost-Schmid complex, the latter being defined only for essentially smooth schemes. We then pass to the study of the functoriality properties of the complexes, the main result being that the homological complex is covariantly functorial for proper morphisms and contravariantly functorial for smooth morphisms, while the cohomological complex is contravariantly functorial for flat morphisms. In both cases, the (co-)homology of the complex is homotopy invariant allowing the definition of an Euler class associated to a vector bundle over a scheme $X$ essentially of finite type over $k$. This Euler class is one of the main motivations for the introduction of Chow-Witt groups by Barge and Morel, but we only spend a few lines on these classes stating only the definition. More properties of these classes can be found in Levine's lectures in the same volume. In the third lecture, we construct a pull-back for morphisms between smooth schemes, and more generally for l.c.i. morphisms between finite type schemes. This pull-back is defined at the level of the (co-)homology and uses homotopy invariance and the deformation to the normal cone construction. This is the only point in these notes where we loose track of what happens at the level of (co-)cycles, mainly because of our use of homotopy invariance. This pull-back allows to define a ring structure at the level of the cohomology of the Rost-Schmid complex. We further study functoriality in this section, stating results on cdh and Nisnevich descent before passing to the base change and projection formulas. At this point, it must be said that our constructions are all based on the theory of cycle modules as studied in \cite{Rost96}, consisting mostly of a refinement of these techniques as in \cite{Morel08}.  

To conclude this introduction, let us say a few words about the history of Chow-Witt groups. They were originally introduced by Barge-Morel (\cite{Barge00}) with the idea to define an Euler class able to tell when a rank $d$ vector bundle $E$ over a smooth affine scheme $X$ actually splits a free bundle of rank one. This program was later fulfilled by Morel in \cite{Morel08}.  The definition of the Chow-Witt groups was based on the one hand on the complex in Milnor $K$-theory and on the other hand on a Gersten-Witt complex defined in the PhD thesis of Schmid (\cite{Schmid98}). Unfortunately, the proof that this complex was indeed a complex was missing at that time. Part of my PhD thesis (\cite{Fasel08a}) was to use another complex to define the Chow-Witt groups, namely the Gersten-Witt complex of Balmer-Walter (\cite{Balmer02}). This complex involved the machinery of derived Witt groups as developed by Balmer, which allowed to give a definition of the Chow-Witt groups and obtain their basic functorial properties (all based on the corresponding functoriality of Witt groups). It was also the occasion to rename the \emph{Chow groups of oriented cycles} of Barge-Morel to Chow-Witt groups. Indeed, the reference to an orientation seemed confusing at that time, simply because Chow-Witt groups are not oriented in the classical sense. 

\subsection*{Acknowledgements} 

I wish to thank the editors, especially Federico Binda, for motivating me to write down these notes.
I'm grateful to Peng Du for carefully reading preliminary versions of this text. This led to the elimination of countless typos and contributed to make the exposition better.

\subsection*{Conventions} We work over a field $k$, assumed to be perfect. The schemes are all $k$-schemes which are of finite type and separated. The fields we consider are all finitely generated field extensions of $k$.

\section{Preliminaries}\label{sec:preliminaries}

\subsection{Milnor-Witt $K$-theory}

Let $F$ be a finitely generated field extension of $k$. Let $A(F)$ be the free (associative, unital) $\ZZ$-graded ring generated in degree $1$ by symbols $[a]$, for $a\in F^{\times}$, and $\eta$ in degree $-1$. The Milnor-Witt $K$-theory ring $\KMW_*(F)$ is the quotient of $A(F)$ generated by the following relations:
\begin{enumerate}
\item $[a]\cdot [1-a]=0$ for $a\neq 0,1$.
\item $[ab]=[a]+[b]+\eta\cdot [a]\cdot [b]$ for $a,b\in F^{\times}$.
\item $\eta\cdot [a]=[a]\cdot \eta$ for $a\in F^{\times}$.
\item $\eta\cdot (\eta\cdot [-1]+2)=0$.
\end{enumerate} 

In the sequel, we will simply denote by $[a_1,\ldots,a_n]$ the product of the $[a_i]$ (for $i=1,\ldots,n$) and by $\eta[a]$ the product $\eta\cdot [a]$.  By construction, the association
\[
F\mapsto \KMW_*(F)
\]
is functorial in $F$. Before stating elementary properties of Milnor-Witt $K$-theory, we first spend a few lines explaining its links with other objects in mathematics. First, note that the ideal generated by $\eta$ is two-sided and that
\[
\KMW_*(F)/(\eta)=\KM_*(F)
\]
where the right-hand side is the Milnor $K$-theory ring as defined in \cite{Milnor69}. The image of $[a_1,\ldots,a_n]$ is usually denoted by $\{a_1,\ldots,a_n\}$. Let now $\GW(F)$ be the Grothendieck-Witt ring of symmetric bilinear forms on $F$ (\cite{Milnor73}). For any $a\in F^{\times}$, we denote by $\langle a\rangle$ the class of the symmetric bilinear form on $F$ defined by $(X,Y)\mapsto aXY$. The following result is fundamental (\cite[Lemma 3.10]{Morel08}).

\begin{lem}\label{lem:degree0}
The map 
\[
\GW(F)\to \KMW_0(F)
\] 
given by $\langle a\rangle\mapsto 1+\eta[a]$ passes to a well-defined isomorphism.
\end{lem}

In view of this result, we will denote by $\langle a\rangle$ the symbol $1+\eta[a]$ of degree $0$. The second relation defining Milnor-Witt $K$-theory then becomes
\[
[ab]=[a]+\langle a\rangle[b]=[b]+\langle b\rangle[a]
\]
for any $a,b\in F^{\times}$. In particular, the fact that $\langle b^2\rangle=\langle 1\rangle$ yields $[ab^2]=[a]+[b^2]$.

Observe also that the symbol $\eta\cdot [-1]+2$ appearing in the fourth relation becomes $\eta[-1]+2=\langle -1\rangle +1=\langle -1,1\rangle$. This is the so-called hyperbolic form on $F^2$. Following the usual conventions, we set $h:=\langle -1,1\rangle$ and we then see that the fourth relation becomes $\eta h=0$.

For $i\geq 1$, we can consider the composite homomorphism
\[
\GW(F)\to \KMW_0(F)\stackrel{\eta^i}\to \KMW_{-i}(F).
\]
Using Lemma \ref{lem:degree0} above, it is not hard to see that this is surjective. By definition, we have $\eta h=0$ and therefore we obtain a surjective homomorphism
\[
\W(F)=\GW(F)/(h)\to \KMW_{-i}(F)
\]
which is in fact an isomorphism (\cite[Lemma 3.10]{Morel08}). The left-hand side is the Witt ring and is the basic object of study when dealing with symmetric bilinear forms on $F$. We have a homomorphism
\[
r:\GW(F)\to \ZZ
\]
which associate to a symmetric bilinear form its rank. The kernel of this homomorphism is the fundamental ideal $\I(F)\subset \GW(F)$. Since $h$ is of rank $2$, we may also consider the homomorphism $\W(F)\to \ZZ/2$ induced by $r$ and observe that the kernel is also $\I(F)$. It is additively generated by the classes of $\langle -1,a\rangle$ in $\W(F)$ (or $\langle -1,a\rangle-h$ in $\GW(F)$). It follows that its powers $\I^n(F)\subset \W(F)$ (for $n\geq 1$) are generated by elements of the form 
\[
\langle\langle a_1,\ldots,a_n\rangle\rangle:=\langle -1,a_1\rangle \cdot \ldots\cdot \langle -1,a_n\rangle
\]
It is convenient to set $\I^n(F)=\W(F)$ for $n\leq 0$. We may also consider the quotient groups
\[
{\overline\I}^n(F):=\I^n(F)/\I^{n+1}(F)
\]
for any $n\in \ZZ$ (by definition, ${\overline\I}^n(F)=0$ for $n\leq -1$).

\begin{lem}\label{lem:forgetful}
For any $n\in\ZZ$, there is a unique (surjective) homomorphism of $\KMW_0(F)$-modules
\[
j_n:\KMW_n(F)\to \I^n(F)
\]
satisfying $[a_1,\ldots,a_n]\mapsto \langle\langle a_1,\ldots,a_n\rangle\rangle$.
\end{lem}

\begin{proof}
If $n\leq 0$, the definition of $j_n$ is obvious. We then suppose that $n\geq 1$.
We may use an alternative presentation of $\KMW_n(F)$ (\cite[Definition 3.3, Lemma 2.4]{Morel08}). Its generators are of the form $\eta^m[a_1,\ldots,a_r]$ with $r-m=n$ and $r,m\in\N$, and the relations  are as follows:
\begin{enumerate}
\item $\eta^m[a_1,\ldots,a_r]=0$ if $a_i+a_{i+1}=1$ for some $i=1,\ldots, r-1$.
\item $\eta^m[a_1,\ldots,xy,\ldots,a_r]=\eta^m[a_1,\ldots,x,\ldots,a_r]+\eta^m[a_1,\ldots,y,\ldots,a_r]+ \\
\eta^{m+1}[a_1,\ldots,x,y,\ldots,a_r]$ for any $x,y\in F^{\times}$.
\item $\eta^{m+2}[a_1,\ldots,a_{i-1},-1,a_{i+1},\ldots,a_r]+2\eta^{m+1}[a_1,\ldots,a_{i-1},a_{i+1},\ldots,a_r]=0$.
\end{enumerate}

We set $j_n(\eta^m[a_1,\ldots,a_r])=\langle\langle a_1,\ldots,a_r\rangle\rangle\in \I^{r}(F)\subset \I^n(F)$, and we check that the three above relations are satisfied. If $F$ is of characteristic different from $2$, this is essentially \cite[Lemma 2.3]{Morel04}. In characteristic $2$, the proof is the same but we spell it out for the convenience of the reader. We may use the presentation of $\GW(F)$ given in \cite[Chapter IV, Lemma 1.1]{Milnor73}. One of the relation reads as
\[
\langle u\rangle +\langle v\rangle=\langle u+v\rangle+\langle(u+v)uv\rangle
\]
for any $u,v\in F^{\times}$ with $u+v\neq 0$. Applying to $a,1-a$, we obtain $\langle a\rangle +\langle 1-a\rangle=\langle 1\rangle+\langle a(1-a)\rangle$ showing that $\langle\langle a,1-a\rangle\rangle=h^2$ in $\GW(F)$. Consequently, the first relation holds in $\I^n(F)$. The third relation holds by definition, and we are left with the second one. It suffices to prove that 
\[
\langle\langle ab\rangle\rangle=\langle\langle a\rangle\rangle+\langle\langle b\rangle\rangle+\langle\langle a,b\rangle\rangle \in \I(F).
\]
Using $\langle u\rangle +\langle -u\rangle=h$ for any $u\in F^\times$, we get 
\[
\langle -1,a\rangle+\langle -1,b\rangle+\langle -1,a\rangle\cdot \langle -1,b\rangle=\langle -1,ab\rangle+3h
\]
whence $\langle\langle a\rangle\rangle+\langle\langle b\rangle\rangle+\langle\langle a,b\rangle\rangle=\langle\langle ab\rangle\rangle$. Finally, it is easy to check that $j_n$ is $\KMW_0(F)$-linear.
\end{proof}

The proof of the lemma also yields the following easy corollary that we state for further reference.

\begin{cor}\label{cor:commuteta}
For any $n\geq 1$, the diagram
\[
\xymatrix{\KMW_{n+1}(F)\ar[r]^-{j_{n+1}}\ar[d]_-\eta & \I^{n+1}(F)\ar[d] \\
\KMW_n(F)\ar[r]_-{j_n} & \I^n(F),}
\]
in which the right-hand vertical map is the inclusion, is commutative.
\end{cor}

 If $F$ is of characteristic different from $2$, recall from \cite[Theorem 4.1]{Milnor69} that for any $n\in \ZZ$ there is a unique homomorphism 
\[
s_n:\KM_n(F)\to {\overline\I}^n(F)
\]
such that $s_n(\{a_1,\ldots,a_n\})=\langle\langle a_1,\ldots,a_n\rangle\rangle \pmod{\I^{n+1}(F)}$. In characteristic $2$, essentially the same proof as in Lemma \ref{lem:forgetful} shows that this homomorphism is well-defined and surjective. By definition, we obtain a commutative square (of $\KMW_0(F)$-modules)
\begin{equation}\label{eqn:fiberproduct}
\xymatrix{\KMW_n(F)\ar[r]^-{j_n}\ar[d] & \I^n(F)\ar[d] \\
\KM_n(F)\ar[r]_-{s_n} & {\overline\I}^n(F).}
\end{equation}

It turns out that this square is actually cartesian (at least in characteristic different from $2$), but this is a highly non trivial fact. To explain this, let us first observe that the homomorphism $j_n$ induces a surjective homomorphism
\[
\overline{j}_n:\KMW_n(F)/(h)\to  \I^n(F)
\]
by $\KMW_0(F)$-linearity.

\begin{theorem}
Suppose that $F$ is of characteristic different from $2$. Then, the diagram (\ref{eqn:fiberproduct}) is Cartesian. 
\end{theorem}

\begin{proof}
We have a commutative diagram of exact sequences
\[
\xymatrix{\KMW_{n+1}(F)\ar[r]^-\eta\ar[d]_-{j_{n+1}}& \KMW_n(F)\ar[r]\ar[d]_-{j_n} & \KM_n(F)\ar[r]\ar[d]_-{s_n} & 0 \\
\I^{n+1}(F)\ar[r] & \I^n(F)\ar[r] & {\overline\I}^n(F)\ar[r] & 0.}
\]
On the other hand, we know from \cite[Theorem 2.4]{Morel04} that $\overline{j}_n$ is an isomorphism (this is a consequence of Voevodsky's affirmation of Milnor conjecture) for each $n\in\Z$. It follows that the map
\[
\KMW_{n+1}(F)/(h)\stackrel{\eta}\to \KMW_{n}(F)/(h)
\]
is injective. Consequently $(h)\cap (\eta)=(\eta h)=0$ and ${\overline\I}^n(F)=\KMW_n(F)/(\eta,h)$. The claim follows.
\end{proof}

\begin{rem}
I don't know if the statement is true if $F$ is of characteristic $2$ (see however \cite[Remark 2.12]{Morel08}).
\end{rem}

To conclude our elementary explorations of the properties of Milnor-Witt $K$-theory, let us observe that the multiplication by $h$ induces a homomorphism 
\[
h_n:\KM_n(F)\to \KMW_n(F)
\]
since $\eta h=0$. This homomorphism can also be described by 
\[
h_n(\{a_1,\ldots,a_n\})=[a_1^2,a_2,\ldots,a_n].
\] 

We conclude this section with some basic properties of Milnor-Witt $K$-theory that we'll use in the sequel (\cite[Lemma 3.7, Corollary 3.8, Lemma 3.14]{Morel08}).

\begin{lem}\label{lem:elementary}
Let $\epsilon:=-\langle -1\rangle\in \KMW_0(F)$. For any $n\in \Z$, let 
\[
n_\epsilon=\begin{cases}  \sum_{i=1}^n \langle (-1)^{i-1}\rangle & \text{if $n>0$.} \\
0 & \text{ if $n=0$.} \\
\epsilon \sum_{i=1}^{-n} \langle (-1)^{i-1}\rangle & \text{ if $n<0$.}\end{cases}
\] 
Then, the following properties are satisfied. 
\begin{enumerate}
\item For any $a\in F^\times$, we have $[a,a]=[-1,a]=[a,-1]$ and $[a,-a]=0$.
\item For any $\alpha\in \KMW_n(F)$ and $\beta\in \KMW_m(F)$, we have $\alpha\beta=\epsilon^{mn}\beta\alpha$.
\item For any $n\in\ZZ$ and $a\in F^\times$, we have $[a^n]=n_{\epsilon}[a]$.
\end{enumerate}
\end{lem}

\subsection{Residue homomorphisms}

Let us now describe the fundamental tool used to define a Gersten-type complex with coefficients in Milnor-Witt $K$-theory: the residue homomorphism. We follow again closely \cite[\S 3]{Morel08}. Let then $F$ be a field and $v:F\to \Z\cup \{-\infty\}$ be a discrete valuation with residue field $\kappa(v)$ and valuation ring $\OO_v$. We choose a uniformizing parameter $\pi=\pi_v$ of $v$. 

\begin{theorem}\label{thm:residue}
There is a unique homomorphism of graded abelian groups
\[
\partial_v^{\pi}:\KMW_{*}(F)\to \KMW_{*-1}(\kappa(v))
\]
of degree $-1$ such that $\partial_v^{\pi}$ commutes with the multiplication by $\eta$ and 
\begin{enumerate}
\item $\partial_v^{\pi}([\pi,u_1,\ldots,u_n])=[\overline u_1,\ldots,\overline u_n],$
\item $\partial_v^{\pi}([u_1,\ldots,u_n])=0,$
\end{enumerate}
for any $u_1,\ldots,u_n\in \OO_v^\times$. We call it the \emph{residue homomorphism}.
\end{theorem}

\begin{proof}
See \cite[Theorem 3.15]{Morel08}.
\end{proof}

\begin{rem}
Granted the existence of $\partial_v^{\pi}$, the unicity is easy. Indeed, let $[a_1,\ldots,a_n]$ be a symbol. Writing $a_i=u_i\pi^{n_i}$ with $u_i\in \OO_v^\times$ and $n_i\in\Z$ for each $i$ and using $[\pi^{n_i}]=(n_i)_{\epsilon}[\pi]$, $[\pi,\pi]=[\pi,-1]$ and the graded commutativity of Milnor-Witt $K$-theory, we can express $[a_1,\ldots,a_n]$ as a sum of symbols of the form $\eta^m[\pi,u_1,\ldots,u_{n+m-1}]$ and $\eta^m[u_1,\ldots,u_{n+m}]$ for some (non constant) $m\in \N$. The images of these symbols are determined by properties 1. and 2. above, yielding unicity.
\end{rem}

\begin{rem}\label{rem:dependsonchoice}
The homomorphism $\partial_v^\pi$ depends not only on the valuation $v$, but also on the choice of the uniformizing parameter $\pi$. Indeed, suppose that $\pi^\prime=u\pi$ for some $u\in \OO_v^\times$. We then have $\partial_v^{\pi}([\pi,-1])=[-1]$ by construction, while
\begin{eqnarray*}
\partial_v^{\pi^\prime}([\pi,-1]) & = & \partial_v^{\pi^\prime}([u^{-1}\pi^\prime,-1]) \\
& = & \partial_v^{\pi^\prime}([u^{-1},-1]+[\pi^\prime,-1]+\eta[u^{-1},\pi^\prime,-1]) \\
& = & \partial_v^{\pi^\prime}([u^{-1},-1]+[\pi^\prime,-1]+\eta[\pi^\prime,u^{-1},-1]) \\
& = & [-1]+\eta[u^{-1},-1] \\
& = & \langle u^{-1}\rangle[-1].
\end{eqnarray*}
In general, $[-1]\neq \langle u^{-1}\rangle[-1]$. For instance, taking $\kappa(v)=\mathbb{R}$ and using the homomorphism $\KMW_1(\real)\to \I(\real)$ we see that $[-1]\neq \langle -1\rangle[-1]$.
\end{rem}

The fact that $\partial_v^\pi$ depends on the choice of the uniformizing parameter is the reason for introducing twisted Milnor-Witt $K$-theory in Section \ref{sec:twistedMW} and the motivation for considering graded line bundles in the next section. To conclude the present section, we state a few fundamental results on the residue homomorphism (\cite[Proposition 3.17, Theorem 3.24]{Morel08}) and introduce the geometric transfer.

\begin{lem}
Let $\alpha\in \KMW_*(F)$ and let $u\in \OO_v^\times$. Then 
\[
\partial_v^{\pi}(\langle u\rangle\alpha)=\langle \overline u\rangle \partial_v^{\pi}(\alpha).
\]
\end{lem}

\begin{theorem}\label{thm:fundamental}
For any $n\in \Z$, there is a split short exact sequence 
\[
0\to \KMW_n(F)\to \KMW_n(F(t))\xrightarrow{\sum\partial_p^p} \bigoplus_{p} \KMW_{n-1}(F(p))\to 0  
\]
where $p$ runs through the set of monic irreducible polynomials of $F[t]$ and $F(p)=F[t]/(p)$. The left-hand homomorphism is induced by the field extension $F\subset F(t)$ and the right-hand homomorphism is the sum of the residues homomorphisms associated to the $p$-adic valuation and the uniformizing parameter $p$.
\end{theorem}

\begin{rem}
There is an obvious retraction of $\KMW_n(F)\to \KMW_n(F(t))$. Indeed, for any $\alpha\in \KMW_n(F(t))$, we can consider $[t]\alpha\in \KMW_{n+1}(F(t))$ and then use the residue homomorphism $\partial_t^t$.
\end{rem}

There is yet another valuation on $F(t)$ that we haven't used so far: the valuation at $\infty$. For any polynomials $f,g\in (F[t]\setminus 0)^2$, we set $v_{\infty}(f/g)=\deg(g)-\deg(f)$. We can choose $-\frac 1t$ as a uniformizing parameter (we'll come back later on the reasons of this choice) and we obtain a residue homomorphism
\[
\partial_{\infty}^{-\frac 1t}:\KMW_{n}(F(t))\to \KMW_{n-1}(F)
\]
for any $n\in\Z$. Following \cite[\S 4.2]{Morel08}, we set the following definition.

\begin{defn}\label{defn:geometrictransfers}
Let $p$ be a monic irreducible polynomial in $F[t]$, the composite
\[
\tau^{F(p)}_F:\KMW_{n-1}(F(p))\subset \bigoplus_{p} \KMW_{n-1}(F(p))\stackrel{s}\to \KMW_n(F(t))\xrightarrow{-\partial_{\infty}^{-\frac 1t}} \KMW_{n-1}(F),
\]
where $s$ is any section of $\sum\partial_p^p$, is called the \emph{geometric transfer} (in weight $n-1$) associated to the extension $F\subset F(p)$.
\end{defn}

\begin{rem}
As the composite $\KMW_n(F)\to \KMW_n(F(t))\xrightarrow{-\partial_{\infty}^{-\frac 1t}} \KMW_{n-1}(F)$ is trivial, the geometric transfer is independent of the choice of the section $s$.
\end{rem}

As a consequence, we obtain a useful unicity property, that is a consequence of the definition of a cokernel. 

\begin{lem}\label{lem:uniquegeom}
The geometric transfers $\tau^{F(p)}_F$ are the unique homomorphisms which make the diagram
\[
\xymatrix{0\ar[r] & \KMW_n(F)\ar[r] & \KMW_n(F(t))\ar[r]^-{\sum\partial_p^p}\ar[d]_-{-\partial_{\infty}^{-\frac 1t}} & \bigoplus_p \KMW_{n-1}(F(p))\ar[r]\ar[ld]^-{\sum \tau_F^{F(p)}} & 0   \\
&  & \KMW_{n-1}(F) & & }
\]
commutative. In other words, there are the unique homomorphisms $f_p:\KMW_{n-1}(F(p))\to \KMW_{n-1}(F)$ with $\sum ( f_p\circ \partial_p^p)+\partial_{\infty}^{-\frac 1t}=0$.
\end{lem}

\begin{rem}\label{rem:Witttransfers}
This unicity property is useful when trying to identify the geometric transfers in the context of symmetric bilinear forms. Indeed, if $n\leq -1$, the geometric transfers give transfer maps 
\[
\W(F(p))\to \W(F)
\]
at the level of Witt groups. Now, transfers for Witt groups are usually obtained via the so-called Scharlau transfer map. For any nontrivial homomorphism of $F$-vector spaces $f_p:F(p)\to F$, one obtains a transfer homomorphism $(f_p)_*:\W(F(p))\to \W(F)$ by considering the composite
\[
V\times V\to F(p)\stackrel{f_p}\to F
\]
for any symmetric bilinear map $V\times V\to F(p)$. In our case, $F(p)$ is a monogeneous field extension, with basis $\{1,t,\ldots,t^{d-1}\}$ where $d=\deg(p)$ and we can consider the homomorphism $f_p:F(p)\to F$ defined by $f(t^i)=0$ for $i=0,\ldots,d-2$ and $f(t^{d-1})=1$. In characteristic different from $2$, we may use \cite[Theorem 4.1]{Scharlau72} to see that the geometric transfers are the ones obtained using the homomorphisms $f_p$ we just defined. The reader can adapt the proof of \emph{loc. cit.} to the case of characteristic $2$ as well.
\end{rem}

\begin{rem}
The main problem with the geometric transfers is that they actually depend on choices. This is one of the main reasons for introducing a slightly different (and more subtle) transfer in Section \ref{sec:twistedMW}.
\end{rem}

\subsection{Graded line bundles}

Let $X$ be a connected scheme, and let $\mathcal P(X)$ be the category of invertible $\OO_X$-modules (or, equivalently, the category of line bundles over $X$). The category $\mathcal P(X)$ is a symmetric monoidal category (in the sense of \cite[Chapter VII, \S 7]{MacLane98}), where the operation is the tensor product of invertible $\OO_X$-modules
\[
(L,N)\mapsto L\otimes_{\OO_X} N,
\]
the unit is the module $\OO_X$ and the associativity and commutativity constraints are the usual ones.
We'll drop the subscript on the tensor product in the sequel. We have three useful isomorphisms, namely the switch isomorphisms $s:L\otimes N\to N\otimes L$ (which satisfies $s\otimes t\mapsto t\otimes s$ on sections), the isomorphism $u:L\otimes L^\vee\simeq \OO_X$ given on sections by $l\otimes \varphi\mapsto \varphi(l)$ and the isomorphism $us:L^\vee\otimes L\simeq \OO_X$. The category $\mathcal G(X)$ of graded line bundles on $X$ is the category whose objects are pairs $(L,a)$ where $L$ is an object of $\mathcal P(X)$ and $a\in \ZZ$ and whose morphisms are of the form
\[
\hom_{\mathcal G(X)}((L,a),(L^\prime,a^\prime))=\begin{cases} \emptyset & \text{ if $a\neq a^\prime$.} \\ \mathrm{Isom}_{\mathcal P(X)}(L,L^\prime) & \text{ if $a=a^\prime$.} \end{cases}
\]
In particular, all morphisms in this category are isomorphisms. We may enrich $\mathcal G(X)$ with a symmetric monoidal structure whose tensor product is defined by 
\[
(L,a)\otimes (L^\prime,a^\prime):=(L\otimes L^\prime,a+a^\prime).
\]
The associativity relation is induced by the associativity relation in $\mathcal P(X)$. The unit is the pair $(\OO_X,0)$ with isomorphisms $(L,a)\otimes (\OO_X,0)\simeq (L,a)$ and $(\OO_X,0)\otimes (L,a)\simeq (L,a)$ induced by the isomorphisms $L\otimes \OO_X\simeq L$ given on sections by $l\otimes a\mapsto al$ (and $a\otimes l\mapsto al$ for the other isomorphism). Finally, the symmetry isomorphism $(L,a)\otimes (L^\prime,a^\prime)\simeq (L^\prime,a^\prime)\otimes (L,a)$ is $(-1)^{aa^\prime}s$ (where $s$ is the switch isomorphism defined above). In $\mathcal G(X)$, each object is invertible. Indeed, we note that $(L,a)\otimes (L^\vee,-a)=(L\otimes L^\vee,0)\stackrel{u}\to (\OO_X,0)$. By symmetry, we see that 
\[
(L^\vee,-a)\otimes (L,a)\simeq (L,a)\otimes (L^\vee,-a)\stackrel{u}\simeq (\OO_X,0)
\]
showing that $(L,a)$ has also a right inverse. Note that the above isomorphism is not induced by $us$, but by $(-1)^aus$. All in all, $\mathcal G(X)$ is a Picard category in the sense of \cite[\S 4]{Deligne87}.

The main purpose of $\mathcal G(X)$ is to understand the properties of the determinant of a vector bundle over $X$. Indeed, let $\mathcal V(X)$ be the category of vector bundles on $X$ (with only isomorphisms as morphisms). there is a functor 
\[
D:\mathcal V(X)\to \mathcal G(X)
\]
given on objects by $V\mapsto (\det V,\mathrm{rk}(V))$ and on morphisms by $f\mapsto \det f$. For each exact sequence $(V)$ of the form
\[
0\to V_1\stackrel{i}\to V_2 \stackrel{p}\to V_3\to 0
\]
with $\mathrm{rk}(V_i)=a_i$, we have an isomorphism 
\[
\varphi_{(V)}:\det V_1\otimes \det V_3\to \det V_2
\]
defined as follows. Locally, the sequence splits, i.e. there exists a monomorphism $s:V_3\to V_2$ such that $ps=\mathrm{id}$. We may then define locally $\varphi_{(V)}$ by 
\[
(v_1\wedge\ldots\wedge v_{a_1})\otimes (w_1\wedge\ldots\wedge w_{a_3})\mapsto v_1\wedge\ldots\wedge v_{a_1}\wedge s(w_1)\wedge\ldots\wedge s(w_{a_3}).
\]
It is straightforward to check that this definition doesn't depend on the choice of the section $s$ and thus that one can glue the local definitions to get a global version. The isomorphism $\varphi_{(V)}$ readily gives an isomorphism $D(V_1)\otimes D(V_3)\simeq D(V_2)$ that we still denote by $\varphi_{(V)}$. This isomorphism is functorial in the sense that if we have a commutative diagram
\[
\xymatrix{0\ar[r] & V_1\ar[r]\ar[d] & V_2\ar[r]\ar[d] & V_3\ar[r]\ar[d] & 0 \\
0\ar[r] & W_1\ar[r] & W_2\ar[r] & W_3\ar[r] & 0}
\]
of exact sequences in which the vertical arrows are isomorphisms, then the diagram
\[
\xymatrix{D(V_1)\otimes D(V_3)\ar[r]^-{\varphi_{(V)}}\ar[d] &D(V_2)\ar[d] \\
D(W_1)\otimes D(W_3)\ar[r]_-{\varphi_{(W)}} & D(W_2)}
\]
in which the vertical arrows are the isomorphisms induced by the vertical arrows above, is also commutative. Moreover, we canonically have that $D(0)=(\OO_X,0)$. This datum satisfies the conditions of \cite[\S 4.3]{Deligne87}. In particular, suppose that we have a sequence of (admissible) monomorphisms of vector bundles
\[
0\subset W\subset V\subset E. 
\]
Then, the following diagram 
\[
\xymatrix{D(W)\otimes D(V/W)\otimes D(E/V)\ar[r]\ar[d] & D(V)\otimes D(E/V)\ar[d] \\
D(W)\otimes D(E/W)\ar[r] & D(E)}
\]
is commutative. Finally, let us mention that the isomorphism $\varphi_{(V)}$ is compatible with the commutativity of $\mathcal G(X)$ in the sense that if we can write $V_2=V_1\oplus V_3$ then the two exact sequences 
\[
\xymatrix{(V) & 0\ar[r]  &V_1\ar[r] & V_2\ar[r] & V_3\ar[r] & 0 }
\]
and 
\[
\xymatrix{(V^\prime) & 0\ar[r]  &V_3\ar[r] & V_2\ar[r] & V_1\ar[r] & 0 }
\]
give a commutative diagram
\[
\xymatrix{D(V_1)\otimes D(V_3)\ar[rr]\ar[rd]_-{\varphi_{(V)}} & & D(V_3)\otimes D(V_1)\ar[ld]^-{\varphi_{(V^\prime)}} \\
 & D(V_2) & }
\]
in which the horizontal arrow is the commutativity isomorphism in $\mathcal G(X)$.

As a result of this general formalism, it is in principle possible to understand (sometimes at the cost of cumbersome computations) the isomorphisms of determinant bundles associated to various exact sequences of vector bundles. This will be particularly useful when dealing with twisted Milnor-Witt $K$-theory. Before spending a section introducing the exact sequences we will use in this work, we conclude this section by mentioning various additional properties of $\mathcal G(X)$. First, we supposed at the beginning of this section that $X$ was connected. When dealing with non-connected schemes, it is convenient to replace the integer appearing in the definition of a graded line bundle with a locally constant integer. The theory and constructions are evidently the same in this slightly more general context. Second, suppose that we have two (connected) schemes $X$ and $Y$ together with a morphism of schemes $f:X\to Y$. Given a graded line bundle $(L,a)$ on $Y$, we may consider the graded line bundle $(f^*L,a)$ on $X$. This defines a functor $f^*:\mathcal G(Y)\to \mathcal G(X)$ that we will use frequently in the rest of these notes.

\subsection{Useful exact sequences}

For any scheme $X$, we denote by $\Omega_{X/k}$ the sheaf of differentials of $X$ over $k$. In case $X$ is smooth and connected, then $\Omega_{X/k}$ is a vector bundle of rank $d_X:=\mathrm{dim}(X)$. We may then consider the graded line bundle $D(\Omega_{X/k})=(\omega_{X/k},d_X)$ in $\mathcal G(X)$ (where $\omega_{X/k}=\det \Omega_{X/k}$). In case $X$ is not connected, then the rank of $\Omega_{X/k}$ is a locally constant integer and we may still consider $D(\Omega_{X/k})$ in $\mathcal G(X)$ in the slightly more general sense we discussed above. We always assume that the schemes we use are connected and let the reader make the necessary changes in the general situation.

Let now $f:X\to Y$ be a smooth morphism of smooth schemes. Then, we have an exact sequence (\cite[II.8]{Hartshorne77})
\[
f^*\Omega_{Y/k}\to \Omega_{X/k}\to \Omega_{X/Y}\to 0
\]
which is also exact on the left as $f$ is smooth. Consequently, we obtain an isomorphism of graded line bundles
\[
D(f^*\Omega_{Y/k})\otimes D(\Omega_{X/Y})\simeq D(\Omega_{X/k})
\]
Using the functor $\mathcal G(Y)\to \mathcal G(X)$ defined in the previous section, we may write the isomorphism as an isomorphism
\begin{equation}\label{eqn:canonical1}
f^*D(\Omega_{Y/k})\otimes D(\Omega_{X/Y})\simeq D(\Omega_{X/k}).
\end{equation}

We will frequently use this isomorphism, or the isomorphisms associated to it after performing elementary operations involving taking inverses and permutations. For instance, we can transform (\ref{eqn:canonical1}) into an isomorphism
\[
(\omega_{X/Y},d_X-d_Y)\simeq (\omega_{X/k},d_X)\otimes (\omega_{Y/k}^\vee,-d_Y)
\]
multiplying first on the left by $(\omega_{Y/k}^\vee,-d_Y)$ and then permuting the latter with $(\omega_{X/k},d_X)$.

Suppose now that we have a Cartesian square of smooth schemes
\[
\xymatrix{X^\prime\ar[r]^-{v}\ar[d]_-g & X\ar[d]^-f \\
Y^\prime\ar[r]_-u & Y}
\]
with $u$ (and thus also $v$) smooth. We then have a canonical isomorphism (\cite[Corollary 4.3]{Kunz86})
\begin{equation}\label{eqn:canonical2}
g^*D(\Omega_{Y^\prime/Y})=D(g^*\Omega_{Y^\prime/Y})\simeq D(\Omega_{X^\prime/X})
\end{equation}

Finally, let us consider the situation where $i:X\to Y$ is a regular embedding of smooth schemes. Let $C_XY$ be the conormal bundle to $X$ in $Y$ and let $N_XY$ be the normal cone to $X$ in $Y$. We then have an exact sequence
\[
C_XY\to i^*\Omega_{Y/k}\to \Omega_{X/k}\to 0
\]
which is in fact exact on the left. Therefore, we obtain a canonical isomorphism
\begin{equation}\label{eqn:canonical3}
D(C_XY)\otimes D(\Omega_{X/k})\simeq i^*D(\Omega_{Y/k})
\end{equation}

As above, we may modify this isomorphism into an isomorphism
\[
D(\Omega_{X/k})\simeq (\det N_XY,d_X-d_Y)\otimes i^*D(\Omega_{Y/k})
\]
by tensoring on the left by $(\det N_XY,d_X-d_Y)$.

\subsection{Twisted Milnor-Witt $K$-theory}\label{sec:twistedMW}

Let $F$ be a field (which is a finitely generated field extension of $k$) and let $(L,a)\in \mathcal G(F)$ be a graded line bundle. We set $L^0$ for the set of nonzero elements of $L$ and we observe that $F^\times$ acts on $L^0$ in an obvious way. It follows that the free abelian group $\ZZ[L^0]$ on the set $L^0$ is endowed with an action of the group algebra $\ZZ[F^\times]$. On the other hand, we know that $\KMW_*(F)$ is a $\KMW_0(F)$-algebra. Besides, we have a map $F^\times \to \KMW_0(F)$ given by $u\mapsto \langle u\rangle$. This induces a ring homomorphism $\ZZ[F^\times]\to \KMW_0(F)$ and it follows that $\KMW_*(F)$ is a $\ZZ[F^\times]$-algebra. The following definition is due to Morel.

\begin{defn}
Let $(L,a)\in \mathcal G(F)$. We define 
\[
\KMW_*(F,L,a):=\KMW_*(F)\otimes_{\ZZ[F^\times]}\Z[L^0]
\]
which we call the \emph{Milnor-Witt $K$-theory of $F$ twisted by $(L,a)$}.
\end{defn}

Even though $a$ doesn't appear in the above definition, it is important to keep track of it. Indeed, we'll sometimes have to compare Milnor-Witt $K$-theory twisted by $(L,a)\otimes (L,a^\prime)$ with Milnor-Witt $K$-theory twisted by $(L,a^\prime)\otimes (L,a)$. We'll then use the commutativity isomorphism in $\mathcal G(X)$, which makes use of both $a$ and $a^\prime$. Another important thing to note is that $\KMW_*(F,L,a)$ is not a ring anymore, but a priori merely a graded abelian group. In fact, it has the structure of a $\KMW_*(F)$-module as we will shortly see. Before, let us observe that the action of $F^\times$ on $L^0$ is simply transitive. It follows that the choice of any non trivial element $l\in L^0$ yields an isomorphism
\[
\KMW_*(F)\simeq \KMW_*(F,L,a)
\] 
given by $\alpha\mapsto \alpha\otimes l$. In view of this fact, we will often denote an element of $\KMW_*(F,L,a)$ by $\alpha\otimes l$ with $\alpha\in \KMW_*(F)$ and $l\in L^0$. Using this notation, we may write the $\KMW_*(F)$-module structure of $\KMW_*(F,L,a)$ as 
\[
\KMW_*(F)\times \KMW_*(F,L,a)\to \KMW_*(F,L,a)
\]
mapping $(\beta,\alpha\otimes l)$ to $(\beta\alpha)\otimes l$. More generally, let $(L,a)$ and $(L^\prime,a^\prime)$ be two graded line bundles over $F$. Then, we define a pairing
\[
\KMW_*(F,L,a)\times \KMW_*(F,L^\prime,a^\prime)\to \KMW_*(F,(L,a)\otimes (L^\prime,a^\prime))
\]
given by $(\alpha\otimes l,\beta\otimes l^\prime)\mapsto (\alpha\beta)\otimes (l\otimes l^\prime)$. We let the reader check that the diagram 
\[
\xymatrix{\KMW_n(F,L,a)\otimes \KMW_m(F,L^\prime,a^\prime)\ar[r]\ar[d] & \KMW_{m+n}(F,(L,a)\otimes (L^\prime,a^\prime))\ar[d] \\
\KMW_m(F,L^\prime,a^\prime) \otimes \KMW_n(F,L,a)\ar[r] & \KMW_{m+n}(F,(L^\prime,a^\prime)\otimes (L,a)),}
\]
in which the left vertical map is the switch homomorphism and the right vertical map is the isomorphism induced by the commutativity rule in $\mathcal G(F)$, is $\langle (-1)^{aa^\prime}\rangle\epsilon^{mn}$-commutative. 

The main idea behind considering twisted Milnor-Witt $K$-theory is the definition of a residue homomorphism as in Theorem \ref{thm:residue} which is independent of the choice of the uniformizing parameter. Let then $F$ be a field with a discrete valuation $v:F\to \ZZ\cup\{-\infty\}$. Let $\OO_v$ be the associated valuation ring with maximal ideal $\mathfrak m_v$ and residue field $\kappa(v)$, and let $(L,a)$ be a graded line bundle on $\OO_v$. We define the \emph{twisted residue map}
\begin{equation}\label{eqn:twistedresidue}
\partial_v:\KMW_*(F,L_F,a)\to \KMW_{*-1}(\kappa(v),((\mathfrak m_v/\mathfrak m_v^2)^*,-1)\otimes (L_{\kappa(v)},a))
\end{equation}
by $\partial_v(\alpha\otimes l)=\partial_v^{\pi}(\alpha)\otimes \overline \pi^*\otimes l$. Here, $(\mathfrak m_v/\mathfrak m_v^2)^*$ is the dual (as $\kappa(v)$-vector spaces) of $\mathfrak m_v/\mathfrak m_v^2$, $L_F$ and $L_{\kappa_v}$ are the respective restrictions of $L$ to $F$ and $\kappa(v)$, $\pi$ is a uniformizing parameter and $\overline \pi^*$ is the dual of the class $\overline \pi$ of $\pi$ in $\mathfrak m_v/\mathfrak m_v^2$.

\begin{lem}
The homomorphism $\partial_v$ is well-defined and doesn't depend on the choice of a uniformizing parameter.
\end{lem}

\begin{proof}
We prove the result for $(L,a)=(\OO_v,0)$ for simplicity. The general case is similar. We are then left to show that the association $\alpha\mapsto \partial_v^{\pi}(\alpha)\otimes \overline \pi^*$ is well-defined and independent of $\pi$. The first assertion is clear as $\partial_v^\pi$ is well-defined. Let now $\pi^\prime$ be another uniformizing parameter. There exists then $u\in \OO_v^\times$ such that $u\pi=\pi^\prime$. Arguing as in Remark \ref{rem:dependsonchoice}, we see that  $\partial_v^{\pi}=\langle u\rangle \partial_v^{\pi^\prime}$. It follows thus from $\overline\pi^*=u^{-1}(\overline\pi^\prime)^*$ that 
\[
\partial_v^{\pi}(\alpha)\otimes \overline \pi^*=\langle u\rangle\partial_v^{\pi^\prime}(\alpha)\otimes \overline \pi^*=\langle u\rangle\partial_v^{\pi^\prime}(\alpha)\otimes u^{-1}(\overline\pi^\prime)^*=\partial_v^{\pi^\prime}(\alpha)\otimes (\overline\pi^\prime)^*.
\]
\end{proof}

Having this twisted residue homomorphism in the pocket, we now turn to the task to define the transfer morphisms for twisted groups. Let then $F/k$ be a finitely generated field extension. We consider the $F[t]$-module $\Omega_{F[t]/k}$ (which is free of rank equal to $\trdeg(F/k)+1$ as $k$ is perfect). For any monic irreducible polynomial $p\in F[t]$, we can consider the $p$-adic valuation and we obtain a twisted residue homomorphism
\[
\partial_p:\KMW_*(F(t),D(\Omega_{F(t)/k}))\to \KMW_{*-1}(F(p),((\mathfrak m_p/\mathfrak m_p^2)^*,-1)\otimes_{F[t]} D(\Omega_{F[t]/k})).
\]
Now, the canonical isomorphism (\ref{eqn:canonical3}) reads in this case as
\[
(\mathfrak m_p/\mathfrak m_p^2,1)\otimes D(\Omega_{F(p)/k})\simeq F(p)\otimes D(\Omega_{F[t]/k})
\]
and we get a canonical isomorphism $D(\Omega_{F(p)/k})\simeq ((\mathfrak m_p/\mathfrak m_p^2)^*,-1)\otimes_{F[t]} D(\Omega_{F[t]/k})$. The above residue map can then be written
\[
\partial_p:\KMW_*(F(t),D(\Omega_{F(t)/k}))\to \KMW_{*-1}(F(p),D(\Omega_{F(p)/k})).
\]
It follows that we get a total residue homomorphism
\[
d:\KMW_*(F(t),D(\Omega_{F(t)/k}))\to \bigoplus_{p}\KMW_{*-1}(F(p),D(\Omega_{F(p)/k}))
\]
where the index set on the right-hand side runs through the set of monic irreducible polynomials in $F[t]$. On the other hand, the sequence $k\subset F\subset F[t]$ gives a version of the canonical isomorphism (\ref{eqn:canonical1}) of the form
\[
D(\Omega_{F/k})\otimes D(\Omega_{F[t]/F})\simeq D(\Omega_{F[t]/k}).
\]
Using the pairing with $\langle 1\rangle\otimes dt\in \KMW_0(F[t],D(\Omega_{F[t]/F}))$ and the field extension $F\subset F(t)$, we then obtain a homomorphism
\[
\KMW_*(F,D(\Omega_{F/k}))\to \KMW_*(F(t),D(\Omega_{F(t)/k})).
\]
\begin{prop}\label{prop:twistedhomotopy}
The sequence
\[
0\to \KMW_*(F,D(\Omega_{F/k}))\to \KMW_*(F(t),D(\Omega_{F(t)/k}))\stackrel{d}\to \bigoplus_{p}\KMW_{*-1}(F(p),D(\Omega_{F(p)/k}))\to 0
\]
is split exact.
\end{prop}

\begin{proof}
Any choice of a generator $l$ of $\omega_{F/k}$ (which is the determinant of $\Omega_{F/k}$ and thus a $F$-vector space of dimension $1$) will give a generator $l\wedge dt$ of $\omega_{F[t]/k}$ (by the isomorphism (\ref{eqn:canonical1})). Using these generators together with the generators $\overline p^*$ of $(\mathfrak m_p/\mathfrak m_p^2)^*$, we obtain an isomorphism of sequences between the sequence of the statement and the sequence of Theorem \ref{thm:fundamental}. The claim follows then immediately.
\end{proof}

\begin{rem}
Of course, one may further twist this sequence with a graded line bundle $(L,a)$ over $F$ and obtain the same result.
\end{rem}

This exact sequence allows us to define transfers as in Definition \ref{defn:geometrictransfers}. 

\begin{defn}
Let $p$ be a monic irreducible polynomial in $F[t]$, we denote by 
\[
\mathrm{Tr}^{F(p)}_F:\KMW_{n-1}(F(p),D(\Omega_{F(p)/k}))\to \KMW_{n-1}(F,D(\Omega_{F/k}))
\]
obtained mimicking the procedure of Definition \ref{defn:geometrictransfers}. We call \emph{canonical transfer} this homomorphism.
\end{defn}

The canonical transfers satisfy the same universal property as the one for geometric transfers stated in Lemma \ref{lem:uniquegeom}. Moreover, they coincide with Morel's absolute transfers as defined in \cite[\S 5.1]{Morel08} (in any characteristic) as evidenced by the following example which treats the separable case.

\begin{ex}
Suppose that the extension $F(p)/F$ is separable. Then, an instructing calculation shows that the transfer map $\mathrm{Tr}^{F(p)}_F$ can be described as follows. Let $l$ be a generator of $\omega_{F/k}$. In view of the isomorphism $\Omega_{F(p)/k}\simeq \Omega_{F/k}\otimes F(p)$, we may consider $l$ as a generator of $\omega_{F(p)/k}$ as well. For any $\alpha\in \KMW_*(F(p))$, we then find the formula
\[
\mathrm{Tr}^{F(p)}_F(\alpha\otimes l)=\tau^{F(p)}_F(\langle p^\prime\rangle\alpha)\otimes l.
\]
\end{ex}

As in the case of geometric transfers, we may have a look at the transfers on Witt groups $\W(F(p),\omega_{F(p)/k})\to\W(F,\omega_{F/k})$ induced by the canonical transfers. In the separable case, the above example, Remark \ref{rem:Witttransfers} and \cite[III,\S 6, Lemme 2]{Serre68} show that this transfer coincides with the one obtained using the trace form. Moreover, the universal property shows that the canonical transfers actually coincide with the ones defined by M. Schmid in his thesis in characteristic different from $2$ (\cite[\S 2]{Schmid98}).

Suppose now that $k\subset F\subset L$ are field extensions such that $L/F$ is finite. We can find a filtration
\[
F=F_0\subset F_1\subset F_2\subset \ldots \subset F_n=L
\]
such that $F_i/F_{i-1}$ is monogeneous for $i=1,\ldots,n$. In that case, we define a transfer homomorphism
\[
\mathrm{Tr}^L_F:\KMW_*(L,D(\Omega_{L/k}))\to \KMW_*(F,D(\Omega_{F/k})) 
\]
as the composite $\mathrm{Tr}^L_F=\mathrm{Tr}^{F_1}_F\circ\ldots\circ\mathrm{Tr}^{F_{n-1}}_{F_{n-2}}\circ \mathrm{Tr}^{L}_{F_{n-1}}$. It is a priori far from clear that this transfer homomorphism is well-defined (i.e. that it doesn't depend on the choice of the filtration above).  This is proved by F. Morel in \cite[Theorem 4.27, \S 5]{Morel08} in full generality. Note however that in case the extension is separable one may use the above discussion (i.e. the fact that canonical transfers are actually induced by the trace form) to deduce that the transfer is well defined. In characteristic different from $2$, we may also deduce the result using \cite[Proposition 2.2.5]{Schmid98}. In any case, we can write the following definition.

\begin{defn}\label{defn:canonicaltransfer}
Let $k\subset F\subset L$ be field extensions such that $L/F$ is finite. We call the homomorphism 
\[
\mathrm{Tr}^L_F:\KMW_*(L,D(\Omega_{L/k}))\to \KMW_*(F,D(\Omega_{F/k})) 
\]
the \emph{canonical transfer}.
\end{defn}

As observed above, it's also possible to further twist $\mathrm{Tr}^L_F$ by a graded line bundle $(L,a)$ in $\mathcal G(F)$. We leave the details to the reader. As a consequence of \cite[Theorem 4.27, \S 5]{Morel08}, we also note that the canonical transfers are functorial in the field extensions of $k$.

\section{The Rost-Schmid complex and Chow-Witt groups}

In the first lecture, we introduced canonical transfers for finite field extensions $F\subset L$ (both field extensions of $k$). These transfers, together with the twisted residue homomorphism (\ref{eqn:twistedresidue}) are the two main ingredients in the construction of the Rost-Schmid complex(es) that we now introduce (always following Morel). 

\subsection{The complexes}

Let $X$ be a (finite type, separated) scheme over $k$ and let $(L,a)$ be a graded line bundle over $X$. For any $j\in\Z$, the Rost-Schmid complex $\tilde{C}(X,j,L,a)$ in weight $j$ is the complex (in homological dimension) whose term in degree $i$ is
\[
\tilde{C}_i(X,j,L,a):=\bigoplus_{x\in X_{(i)}} \KMW_{j+i}(k(x),D(\Omega_{k(x)/k})\otimes (L,a)).
\]
Let $x\in X_{(i)}$ and $y\in X_{(i-1)}$. The differential
\[
d_y^x:\KMW_{j+i}(k(x),D(\Omega_{k(x)/k})\otimes (L,a))\to \KMW_{j+i-1}(k(y),D(\Omega_{k(y)/k})\otimes (L,a))
\]
is defined as follows. If $y\not\in  \overline {\{x\}}$, then we set $d_y^x=0$. In the other case, let $Z$ be the normalization of $\overline {\{x\}}$. We have a finite morphism $f:Z\to X$. Removing if necessary points of codimension $\geq 2$ in $Z$ and their images in $X$, we may assume that we have a finite morphism $Z\to X$ and that $Z$ is smooth. In particular, we can consider its module of differentials $\Omega_{Z/k}$ (which is locally free of rank $i$). For any $z\in Z_{(i-1)}$ we then have a residue homomorphism (\ref{eqn:twistedresidue})
\[
d^x_z:\KMW_{j+i}(k(x),D(\Omega_{k(x)/k})\otimes (L,a))\to \KMW_{j+i-1}(k(z),((\mathfrak m_z/\mathfrak m_z^2)^*,-1)\otimes D(\Omega_{Z/k})\otimes (L,a))
\]
Using the canonical isomorphism (\ref{eqn:canonical3})
\[
((\mathfrak m_z/\mathfrak m_z^2)^*,-1)\otimes D(\Omega_{Z/k})\simeq D(\Omega_{k(z)/k})
\]
we obtain a homomorphism
\[
d^x_z:\KMW_{j+i}(k(x),D(\Omega_{k(x)/k})\otimes (L,a))\to \KMW_{j+i-1}(k(z),D(\Omega_{k(z)/k})\otimes (L,a))
\]
If $z\mapsto y$ under the map $f:Z\to X$, the field extension $k(y)\subset k(z)$ is finite and there is a canonical transfer homomorphism
\[
\KMW_{j+i-1}(k(z),D(\Omega_{k(z)/k})\otimes (L,a))\to \KMW_{j+i-1}(k(y),D(\Omega_{k(y)/k})\otimes (L,a)).
\]
Summing the composite maps over all $z\in Z$ mapping to $y$, we finally obtain a homomorphism
\[
d_y^x:\KMW_{j+i}(k(x),D(\Omega_{k(x)/k})\otimes (L,a))\to \KMW_{j+i-1}(k(y),D(\Omega_{k(y)/k})\otimes (L,a))
\]
as required. For a given $x$ and a given $\alpha\in \KMW_{j+i}(k(x),D(\Omega_{k(x)/k})\otimes (L,a))$, there are only finitely many $y\in X_{(i-1)}$ such that $d_y^x(\alpha)\neq 0$ and we then obtain a well defined homomorphism
\[
d_i: \tilde{C}_i(X,j,L,a)\to \tilde{C}_{i-1}(X,j,L,a).
\]
In short, we have a graded abelian group $\tilde{C}(X,j,L,a)$ together with a homomorphism $d$ of degree $-1$, i.e. a pair $(\tilde{C}(X,j,L,a),d)$. It is not clear at all that $d_{i-1}d_i=0$. We will shortly state the theorem summarizing the situation, but before we focus on the Rost-Schmid complex in the special situation where $X$ is smooth (and connected). In this case, we can consider the graded line bundle $D(\Omega_{X/k})$ on $X$ and the pair $(\tilde{C}(X,j,D(\Omega_{X/k})^{-1}),d)$ where $D(\Omega_{X/k})^{-1}$ is the inverse of $D(\Omega_{X/k})$ in $\mathcal G(X)$ (more explicitly, $D(\Omega_{X/k})^{-1}=(\omega_{X/k}^\vee,-\mathrm{dim}(X))$). Using the canonical isomorphism (\ref{eqn:canonical3}), we see that the graded components are of the form
\[
\tilde{C}_i(X,j,D(\Omega_{X/k})^{-1})= \bigoplus_{x\in X_{(i)}} \KMW_{j+i}(k(x),D(\mathfrak m_x/\mathfrak m_x^2)^{-1}).
\]
Setting $\mathrm{dim}(X):=d_X$, we may rewrite the above term as
\[
\tilde{C}_i(X,j,D(\Omega_{X/k})^{-1})= \bigoplus_{x\in X^{(d_X-i)}} \KMW_{j+i}(k(x),D(\mathfrak m_x/\mathfrak m_x^2)^{-1}).
\]
Now, we can consider the graded abelian group $C(X,j)$ whose component of degree $i$ is of the form 
\[
C(X,j)^i:=\bigoplus_{x\in X^{(i)}} \KMW_{j-i}(k(x),D(\mathfrak m_x/\mathfrak m_x^2)^{-1})
\]
and we have $\tilde{C}_i(X,j,D(\Omega_{X/k})^{-1})=C(X,j+d_X)^{d_X-i}$. It follows that we obtain a homomorphism $d:C(X,j)\to C(X,j)$ of degree $+1$. 

\begin{rem}
Of course, we may also consider the complex $C(X,j,L,a)$ for any graded line bundle $(L,a)$ over $X$.
\end{rem}

\begin{theorem}
Let $X$ be a finite type $k$-scheme, $j\in\ZZ$ be an integer and $(L,a)\in \mathcal G(X)$ be a graded line bundle over $X$. Then, $(\tilde{C}(X,j,L,a),d)$ is a (chain) complex provided $k$ is of characteristic different from $2$. If $X$ is smooth, $(C(X,j,L,a),d)$ is a (cochain) complex for $k$ of any characteristic.
\end{theorem}

\begin{proof}
If $X$ is smooth, this is \cite[Theorem 5.31]{Morel08}. In the general case, the proof can be found in \cite[\S 4.2.8]{Deglise17}.
\end{proof}

\begin{rem}\label{rem:char2}
The assumption that $k$ is of characteristic different from $2$ in the proof that $(\tilde{C}(X,j,L,a),d)$ is a complex is artificial, and will be removed in \cite{Calmes17b}. Further, one can relax the assumption on $X$ by only supposing that $X$ is essentially of finite type. 
\end{rem}

\begin{defn}
Let $X$ be a finite type scheme over $k$. We call $\tilde{C}(X,j,L,a)$ the \emph{homological Rost-Schmid complex twisted by $(L,a)$}. We write $\H_i(X,j,L,a)$ for the homology groups of $\tilde{C}(X,j,L,a)$. In case $X$ is smooth, we call $C(X,j,L,a)$ the \emph{cohomological Rost-Schmid complex twisted by $(L,a)$} and write $\H^i(X,j,L,a)$ for its cohomology groups. If $Y\subset X$ is a closed subset, we can also consider the subcomplexes of both complexes formed by points on $Y$. We respectively denote by $\tilde{C}_Y(X,j,L,a)$ and $C_Y(X,j,L,a)$ these complexes.
\end{defn}

If $X_{\mathrm{red}}$ denotes the reduction of $X$, then we can observe that $\tilde{C}(X,j,L,a)=\tilde{C}(X_{\mathrm{red}},j,L,a)$ by very definition. Further, we can also observe that if $i:Y\subset X$ is a closed subscheme, then we also have an identification
\[
\tilde{C}(Y,j,i^*(L,a))=\tilde{C}_Y(X,j,L,a).
\]
We finally come to the definition of Chow-Witt groups. 

\begin{defn}\label{defn:homoCW}
Let $X$ be a scheme (essentially of finite type over $k$, separated). For any $i\in\N$ and any graded line bundle $(L,a)$ over $X$, we set 
\[
\CHW_i(X,L,a)=\H_i(X,-i,L,a).
\]
We call this homology group the \emph{homological Chow-Witt group of $i$-dimensional cycles twisted by $(L,a)$}. For simplicity, we also sometimes simply say \emph{homological $i$-th Chow-Witt group} or \emph{Chow-Witt group}. If $(L,a)$ is trivial, we omit it from the notation.
\end{defn}

\begin{defn}\label{def:cohomoCW}
Let $X$ be an essentially smooth $k$-scheme. For any $i\in\N$ and any graded line bundle $(L,a)$ over $X$, we set 
\[
\CHW^i(X,L,a)=\H^i(X,i,L,a).
\]
We call this cohomology group the \emph{cohomological Chow-Witt group of $i$-codimension cycles twisted by $(L,a)$}. As above, we also sometimes simply say \emph{cohomological $i$-th Chow-Witt group} or \emph{Chow-Witt group} and we omit $(L,a)$ from the notation in case it is trivial. If $Y\subset X$ is a closed subset, we may also consider the group
\[
\CHW^i_Y(X,L,a)=\H^i(C_Y(X,i,L,a)).
\] 
and call it the \emph{cohomological Chow-Witt group of $i$-codimension cycles supported on $Y$ and twisted by $(L,a)$}, or simply \emph{Chow-Witt group supported on $Y$}. 
\end{defn}

\subsection{Basic properties}\label{sec:basicprop}

We now state some basic properties of the Chow-Witt groups. First note that $\CHW_i(X,L,a)=0$ and $\CHW^i(X,L,a)=0$ for $i>\mathrm{dim}(X)$ or $i<0$. If $X$ is smooth of (constant) dimension $d_X$, the above identification of complexes yields an isomorphism
\[
\CHW_i(X,D(\Omega_{X/k})^{-1}\otimes(L,a))\simeq \CHW^{d_X-i}(X,L,a).
\] 
Next, observe that we may replace Milnor-Witt $K$-theory by Milnor $K$-theory in the definition of the Rost-Schmid complex. In a precise manner, this amounts to consider the quotient of each term by $\eta$. Since the differential maps commute with $\eta$, it follows that we get a differential at the level of the quotient complex. By definition, it coincides with the differential defined in \cite[\S2, (2.1.0)]{Rost96} and consequently we obtain homomorphisms
\[
\CHW_i(X,L,a)\to \CH_i(X)
\] 
for any $i\in\N$ (the same applies for the cohomological versions). In general, this map is neither injective, nor surjective as shown by \cite[Corollary 11.8]{Fasel09d}. Next, suppose that $\iota:Y\subset X$ is a closed subset. We may see $Y$ as a scheme by endowing it with its reduced structure and we know that $\tilde{C}(Y,j,\iota^*(L,a))=\tilde{C}_Y(X,j,L,a)$. Let $U\subset X$ be the complement of $Y$ in $X$ with embedding $u:U\to X$. It is an easy exercise to show that we have an exact sequence of complexes
\[
0\to \tilde{C}(Y,j,\iota^*(L,a))\to \tilde{C}(X,j,L,a)\to \tilde{C}(U,j,u^*(L,a))\to 0
\]
and consequently a long exact sequence
\[
\ldots \to \CHW_i(Y,\iota^*(L,a))\to \CHW_i(X,L,a)\to \CHW_i(U,u^*(L,a))\to \ldots
\]
in homology. In contrast with the situation in the case of Chow groups, the right-hand arrow $ \CHW_i(X,L,a)\to \CHW_i(U,u^*(L,a))$ is in general not surjective (\cite{Fasel09d}). In case the scheme $X$ is smooth, and $\iota:Y\subset X$ is closed, we obtain the same sequence, with the distinction that $\CHW_i(Y,\iota^*(L,a))$ can be compared to the relevant cohomological group only if $Y$ is also smooth. Letting $r=d_X-d_Y$ be the (constant) codimension of $Y$ in $X$ and considering the subcomplex $C_Y(X,j,L,a)\subset C(X,j,L,a)$ of points supported on $Y$, we have identifications
\begin{eqnarray*}
C_Y(X,j,L,a) & = & \tilde{C}_Y(X,j-d_X,D(\Omega_{X/k})^{-1}\otimes (L,a)) \\
& = & \tilde{C}(Y,j-d_X,\iota^*D(\Omega_{X/k})^{-1}\otimes \iota^*(L,a)) \\
& = & C(Y,j+d_Y-d_X,D(\Omega_{Y/k})\otimes \iota^*D(\Omega_{X/k})^{-1}\otimes \iota^*(L,a))
\end{eqnarray*}
Thus, the above exact sequence reads
\[
\ldots \to \CHW^{i-r}(Y,D(\Omega_{Y/k})\otimes \iota^*D(\Omega_{X/k})^{-1}\otimes \iota^*(L,a))\to \CHW^i(X,L,a)\to \CHW^i(U,u^*(L,a))\to \ldots
\]
in this case. One may moreover use the canonical isomorphism (\ref{eqn:canonical1}) to identify $D(\Omega_{Y/k})\otimes \iota^*D(\Omega_{X/k})^{-1}$ with $D(\Omega_{X/Y})$.

\subsection{Push-forwards}

The purpose of this section is to show that given a proper morphism $f:X\to Y$ of finite type schemes over $k$ there exists a morphism of complexes
\[
f_*:\tilde{C}(X,j,f^*(L,a))\to \tilde{C}(Y,j,L,a)
\]
for any $j\in\Z$ and any graded line bundle $(L,a)$ on $Y$. The association $f\mapsto f_*$ is functorial (i.e. respects compositions and identities) and therefore homological Chow-Witt groups define a covariant functor from the category of finite type $k$-schemes (with projective morphisms) to the category of abelian groups. The definition of $f_*$ is quite straightforward, given the existence of transfer maps.  Indeed, let $x\in X$ and let $y=f(x)\in Y$ be such that $x\in X_{(i)}$ and $y\in Y_{(i^\prime)}$. If the field extension $k(y)\subset k(x)$ is infinite, then we define 
\[
(f_*)^x_y:\KMW_{j+i}(k(x),D(\Omega_{k(x)/k})\otimes (L,a))\to \KMW_{j+i^\prime}(k(y),D(\Omega_{k(y)/k})\otimes (L,a)) 
\]
to be trivial. If the extension $k(y)\subset k(x)$ is finite, then $i=i^\prime$ and we define 
\[
(f_*)^x_y:\KMW_{j+i}(k(x),D(\Omega_{k(x)/k})\otimes (L,a))\to \KMW_{j+i}(k(y),D(\Omega_{k(y)/k})\otimes (L,a)) 
\]
to be the canonical transfer $\mathrm{Tr}^{k(x)}_{k(y)}$ of Definition \ref{defn:canonicaltransfer}. Summing up these maps, we obtain a morphism of graded abelian groups 
\[
f_*:\tilde{C}(X,j,f^*(L,a))\to \tilde{C}(Y,j,L,a)
\]
this is evidently functorial since transfer maps are functorial. 

\begin{theorem}\label{thm:properpf}
Let $f:X\to Y$ be a morphism of $k$-schemes essentially of finite type. Then, 
\[
f_*:\tilde{C}(X,j,f^*(L,a))\to \tilde{C}(Y,j,L,a)
\]
is a morphism of complexes provided one of the following conditions hold:
\begin{enumerate}
\item $f$ is finite and both $X$ and $Y$ are essentially smooth.
\item $f$ is proper and $k$ is of characteristic different from $2$.
\end{enumerate}
\end{theorem}

\begin{proof}
The result under the first assumption is proved in \cite[Corollary 5.30]{Morel08}. For the proof of the statement under the second assumption, we refer to \cite[\S 4.2.5, Theorem 4.2.7]{Deglise17}.
\end{proof}

\begin{rem}
As in Remark \ref{rem:char2}, the assumption that $k$ is of characteristic different from $2$ should be superfluous. We hope to lift it in \cite{Calmes17b}. 
\end{rem}

As an obvious consequence, we see that for any $i\in\N$, the homological Chow group $\CHW_i$ is a functor from the category of finite-type schemes (with only projective morphisms as morphisms) to the category of abelian groups. Moreover, this allows to define a refinement of the usual degree map on Chow groups. Suppose that $X$ is a projective $k$-scheme with structural morphism $p:X\to \Spec{k}$. We then have a push-forward map 
\[
p_*:\CHW_0(X)\to \CHW_0(k)=\KMW_0(k).
\]
and a commutative diagram
\[
\xymatrix{\CHW_0(X)\ar[r]^-{p_*}\ar[d] & \KMW_0(k)\ar[d] \\
\CH_0(X)\ar[r]_{\mathrm{deg}} & \ZZ}
\]
in which the vertical homomorphisms are the ones from Chow-Witt groups to Chow groups and the bottom horizontal one is the degree map. Both vertical homomorphisms are surjective and we see that $p_*$ is a refinement of the degree map. We denote it $\widetilde{\mathrm{deg}}$ below and call it the \emph{Milnor-Witt degree} of $X$. Virtually, all the questions that are interesting for $\mathrm{deg}$ are also interesting for $\widetilde{\mathrm{deg}}$. For instance, it is clear (by functoriality) that if $X$ has a rational point, then $\widetilde{\mathrm{deg}}$ is surjective.  The converse statement is also interesting: Suppose that $\widetilde{\mathrm{deg}}$ is surjective, then does $X$ have a rational point? Funnily, the answer is that $\widetilde{\mathrm{deg}}$ is no better than $\mathrm{deg}$ in this respect, i.e. it doesn't detect rational points. For more information on this story, we refer to \cite{Asok11b}. This question set apart, we may observe that the image of $\widetilde{\mathrm{deg}}$ is harder to compute than the image of the degree. Indeed, $\KMW_0(k)$ is not a PID and the only a priori structure of the image is that it is a $\KMW_0(k)$-submodule of $\KMW_0(k)$. We refer to \cite{Bhatwadekar14} for some computations in this direction, and we note that it would be interesting to compute some Milnor-Witt degrees of hypersurfaces in $\mathbb{P}^n$.

Finally, let's state the cohomological version of Theorem \ref{thm:properpf} which is obtained via identification of the cohomological and homological complexes.

\begin{theorem}
Let $f:X\to Y$ be a morphism of smooth $k$-schemes, and let $d=\mathrm{dim}(Y)-\mathrm{dim}(X)$. Then the morphism of degree $d$
\[
f_*:C(X,j,D(\omega_{X/k})\otimes f^*(L,a))\to C(Y,j+d,D(\Omega_{Y/k})\otimes (L,a))
\]
is a morphism of complexes provided one of the following conditions hold:
\begin{enumerate}
\item $f$ is finite.
\item $f$ is proper and $k$ is of characteristic different from $2$.
\end{enumerate}
\end{theorem}

\subsection{Flat pull-backs}

We now pass to the notion of pull-back for Chow-Witt groups. As in the case of Chow groups, the construction of the pull-backs is more delicate than the one for push-forwards. The subtle point here is that we need a notion of a ``quadratic" (or more precisely ``symmetric bilinear") length in order to define explicitly the pull-back for flat morphisms.  Our aim is then to construct a morphism of complexes
\[
f^*:C(X,j,L,a)\to C(Y,j,f^*(L,a))
\]
for any flat morphism $f:Y\to X$ between (essentially) smooth schemes and any graded line bundle $(L,a)$ over $X$. We start by observing that the basic bricks of the complex $C(X,j,L,a)$ are of the form $\KMW_j(k(x),D(\mathfrak m_x/\mathfrak m_x^2)^{-1}\otimes (L,a))$. We omit $(L,a)$ in the sequel, letting the reader making the necessary changes to keep track of this graded line bundle. The idea behind what follows is that the Milnor-Witt $K$-theory of $k(x)$ is deeply linked with Hermitian $K$-theory (aka higher Grothendieck-Witt groups) of the regular local ring $\OO_{X,x}$.

Let then $x\in X^{(i)}$ and consider the category $\OO_{X,x}^{\mathrm{fl}}$ of finite length $\OO_{X,x}$-modules, which is an abelian category. To any finite length module $M$, we can associate a finite length module $M^{\sharp}:=\mathrm{Ext}^i_{\OO_{X,x}}(M,\OO_{X,x})$, obtaining an exact functor
\[
^{\sharp}:(\OO_{X,x}^{\mathrm{fl}})^{\mathrm{op}}\to \OO_{X,x}^{\mathrm{fl}}
\]
Moreover, there is a canonical isomorphism $\varpi:\mathrm{id}\to ^{\sharp\sharp}$ (\cite[\S 6]{Balmer02}) turning $\OO_{X,x}^{\mathrm{fl}}$ into an exact category with duality in the sense of \cite[Definition 2.1]{Schlichting09}. To this category, we may then associate its Grothendieck-Witt group $\GW^{\mathrm{fl}}(\OO_{X,x})$ in the sense of \cite[\S 2.2]{Schlichting09}. Roughly speaking, this group is the Grothendieck group of the set of isometry classes of pairs $(M,\varphi)$ (where $M$ is a finite length module and $\varphi:M\to M^{\sharp}$ is a symmetric isomorphism) endowed with the orthogonal sum as operation, modulo an extra relation identifying a pair $(M,\varphi)$ with a totally isotropic submodule with an hyperbolic module. We now explain how to compute $\GW^{\mathrm{fl}}(\OO_{X,x})$. Let $\OO_{X,x}^0\subset \OO_{X,x}^{\mathrm{fl}}$ be the subcategory of semi-simple objects. It inherits the duality ${}^\sharp$ from the latter and is in turn an exact category with duality. The induced homomorphism between Grothendieck-Witt groups 
\[
\GW(\OO_{X,x}^0)\to \GW^{\mathrm{fl}}(\OO_{X,x})
\]
is actually an isomorphism by \cite[Theorem 6.10]{Quebbemann79}. We now identify the left-hand side with a more familiar object. Let $V$ be a (finite dimensional) $k(x)$-vector space. We claim that there is an isomorphism of $k(x)$-vector spaces
\[
\hom_{k(x)}(V,\mathrm{Ext}^i_{\OO_{X,x}}(k(x),\OO_{X,x}))\to \mathrm{Ext}^i_{\OO_{X,x}} (V,\OO_{X,x}).
\]
Indeed, one can choose a projective $\OO_{X,x}$-resolution of $k(x)$
\[
0\to P_{i}\to \ldots\to P_1\to P_0\to k(x)\to 0
\]
Dualizing (and setting $P_j^\vee=\hom_{\OO_{X,x}}(P,\OO_{X,x})$), we obtain a projective resolution
\[
0\to P_0^\vee\to P_1^\vee\to \ldots\to P_i^\vee\to \mathrm{Ext}^i_{\OO_{X,x}}(k(x),\OO_{X,x})\to 0
\]
We can pull-back this exact sequence along any homomorphism of $\OO_{X,x}$-modules $V\to \mathrm{Ext}^i_{\OO_{X,x}}(k(x),\OO_{X,x})$ and get a projective resolution of $V$ as an $\OO_{X,x}$-module, which is nothing else than an extension of $V$ by $P_0^\vee$. Now, it is clear that we can choose $P_0=\OO_{X,x}$ and we obtain a map
\[
\hom_{k(x)}(V,\mathrm{Ext}^i_{\OO_{X,x}}(k(x),\OO_{X,x}))\to \mathrm{Ext}^i_{\OO_{X,x}} (V,\OO_{X,x}).
\]
as required. We let the reader check that it is $k(x)$-linear, functorial in $V$ and that it respects direct sums of $k(x)$-vector spaces. Since it is obviously an isomorphism for $V=k(x)$, it is an isomorphism. In more sophisticated terms, the above homomorphism defines a duality preserving functor between the category of $k(x)$-vector spaces, endowed with the duality $\hom_{k(x)}(\_,\mathrm{Ext}^i_{\OO_{X,x}}(k(x),\OO_{X,x}))$ and the category $\OO_{X,x}^0$. This functor is an equivalence, and therefore we obtain an isomorphism of Grothendieck-Witt groups 
\[
\GW(k(x),\mathrm{Ext}^i_{\OO_{X,x}}(k(x),\OO_{X,x}))\simeq \GW(\OO_{X,x}^0).
\]
Finally, we can state the following proposition, which is the basis of the definition of flat pull-backs.

\begin{prop}\label{prop:localmult}
Let $X$ be an essentially smooth scheme and let $x\in X^{(i)}$. Then, we have a canonical isomorphism 
\[
\KMW_0(k(x),D(\mathfrak m_x/\mathfrak m_x^2)^{-1})\to \GW^{\mathrm{fl}}(\OO_{X,x}).
\] 
\end{prop}

\begin{proof}
First, we know from Lemma \ref{lem:degree0} that we have an isomorphism 
\[
\KMW_0(k(x))\to \GW(k(x))
\]
such that $\langle u\rangle \mapsto \langle u\rangle$ for any $u\in k(x)^\times$. For any $k(x)$-vector space $L$ of dimension $1$, we define a map
\[
\KMW_0(k(x))\otimes_{\ZZ[k(x)^\times]}\ZZ[L^0]\to \GW(k(x),L)
\]
by $\langle u\rangle \otimes l\mapsto \varphi_{ul}$, where $\varphi_l:k(x)\times k(x)\to L$ is defined by $\varphi_l(x,y)=xyl$. We let the reader check (for instance, using the isomorphisms between untwisted and twisted groups induced by the choice of $l$) that this map is in fact an isomorphism of abelian groups. Consequently, we have a sequence of isomorphisms
\[
\KMW_0(k(x),a,\mathrm{Ext}^i_{\OO_{X,x}}(k(x),\OO_{X,x}))\simeq \GW(k(x),\mathrm{Ext}^i_{\OO_{X,x}}(k(x),\OO_{X,x}))\simeq \GW^{\mathrm{fl}}(\OO_{X,x})
\]
for any integer $a\in \ZZ$. The claim then follows using the canonical isomorphism $(\mathfrak m_x/\mathfrak m_x^2)^*\simeq \mathrm{Ext}^i_{\OO_{X,x}}(k(x),\OO_{X,x})$ deduced from \cite[A X.165, Proposition 9]{Bourbaki80} and setting $a=-1$.
\end{proof}

One of the interesting features about Grothendieck-Witt groups of finite length modules is that they are functorial for flat morphisms between rings of the same Krull dimension. Indeed, let $A,B$ be regular local rings of the same Krull dimension $d$, and let $f:A\to B$ be a flat homomorphism of local rings. We have a functor $f^*:A^{\mathrm{fl}}\to B^{\mathrm{fl}}$ between the respective categories of finite length modules given by $M\mapsto B\otimes_A M$. For any projective $A$-module $P$, we have isomorphisms
\[
B\otimes \mathrm{Hom}_A(P,A)\simeq \mathrm{Hom}_B(B\otimes P,B)
\]
and it follows that we have an isomorphism $B\otimes_A\mathrm{Ext}^d_A(M,A)\simeq \mathrm{Ext}^d_B(B\otimes_A M,B)$ which is functorial in $M$. Consequently, we obtain a homomorphism 
\[
f^*:\GW^{\mathrm{fl}}(A)\to \GW^{\mathrm{fl}}(B)
\]
which is actually functorial. 

Let now $X$ and $Y$ be essentially smooth schemes and let $f:Y\to X$ be a flat morphism. Let $x\in X^{(i)}$ and let $y\in Y^{(i)}$ be such that $f(y)=x$ (in fact, $y$ is a minimal point of the fiber of $x$). The induced homomorphism
\[
f:\OO_{X,x}\to \OO_{Y,y}
\]
is then a flat homomorphism of regular local rings and we get an induced homomorphism $f^*$ on the respective Grothendieck-Witt groups of finite length modules. In view of Proposition \ref{prop:localmult}, we obtain a homomorphism $m(x,y,f)$ of the form
\[
\KMW_0(k(x),D(\mathfrak m_x/\mathfrak m_x^2)^{-1})\simeq \GW^{\mathrm{fl}}(\OO_{X,x})\stackrel {f^*}\to \GW^{\mathrm{fl}}(\OO_{Y,y})\simeq \KMW_0(k(y),D(\mathfrak m_y/\mathfrak m_y^2)^{-1})
\]
which we call \emph{local multiplicity along $f$}. This terminology can be justified as follows. If $R$ is a (noetherian) regular local ring with maximal ideal $\mathfrak m$, then one can consider its completion $\hat R$ with respect to the $\mathfrak m$-adic valuation. This is still a regular local ring with the same residue field $\kappa(m)$ as $R$ and the extension $R\to \hat R$ is flat. As a consequence of Proposition \ref{prop:localmult}, we see that the induced map
\[
\GW^{\mathrm{fl}}(R)\to \GW^{\mathrm{fl}}(\hat R)
\]
is an isomorphism. We may use the Cohen structure theorem to find a presentation $\hat R=\kappa(x)[[x_1,\ldots,x_d]]$ and finally we see that what we are looking at is in fact the group $\GW^{\mathrm{fl}}(\kappa(x)[[t_1,\ldots,t_d]])$. A flat homomorphism as above then induces a homomorphism of the form
\[
\GW^{\mathrm{fl}}(k(x)[[t_1,\ldots,t_i]])\to \GW^{\mathrm{fl}}(k(y)[[t_1,\ldots,t_i]])
\]
which is in some sense the multiplication by the (motivic) Brouwer degree of the homomorphism $k(x)[[t_1,\ldots,t_i]]\to k(y)[[t_1,\ldots,t_i]]$.

The local multiplicity along $f$ allows us to define homomorphisms 
\[
(f^*)_x^y:\KMW_{j-i}(k(x),D(\mathfrak m_x/\mathfrak m_x^2)^{-1})\to \KMW_{j-i}(k(y),D(\mathfrak m_y/\mathfrak m_y^2)^{-1})
\]
for any $j\in\ZZ$ as follows. Let $\iota:\KMW_{j-i}(k(x))\to \KMW_{j-i}(k(y))$ be the homomorphism induced by the field extension. We set 
\[
(f^*)_x^y(\alpha\otimes \overline x_1^*\wedge\ldots\wedge \overline x_i^*)=\iota(\alpha)\cdot m(x,y,f)(\langle 1\rangle\otimes x_1^*\wedge\ldots\wedge \overline x_i^*)
\]
for any generators $\overline x_1,\ldots,\overline x_i$ of $\mathfrak m_x/\mathfrak m_x^2$. Of course, we may further twist by any graded line bundle $(L,a)$ and we leave as usual the details to the reader.

\begin{ex}\label{ex:smoothpb}
Suppose that $f:Y\to X$ is smooth and that $x\in X^{(i)}$ and $y\in Y^{(i)}$ are such that $f(y)=x$. In that case, $f$ induces an isomorphism $\overline f:(\mathfrak m_x/\mathfrak m_x^2)\otimes k(y)\to \mathfrak m_y/\mathfrak m_y^2$ and the local multiplicity $m(x,y,f)$ is easily seen to satisfy 
\[
m(x,y,f)(\langle 1\rangle\otimes \overline x_1^*\wedge\ldots\wedge \overline x_i^*)=\langle 1\rangle\otimes \overline f(\overline x_1)^*\wedge\ldots\wedge  \overline f(\overline x_i)^*
\]
showing that in that case $(f^*)_x^y$ is essentially determined  by $\iota:\KMW_{j-i}(k(x))\to \KMW_{j-i}(k(y))$.
\end{ex}

All in all, we see that if $f:Y\to X$ is a flat morphism, we can define a pull-back homomorphism
\[
f^*:C(X,j,L,a)\to C(Y,j,f^*(L,a))
\]
by summing up the homomorphisms $(f^*)_x^y$. The association $f\mapsto f^*$ is functorial.

\begin{theorem}
Let $f:Y\to X$ be a flat morphism of essentially smooth $k$-schemes and let $(L,a)$ be a graded line bundle over $X$. Then, 
\[
f^*:C(X,j,L,a)\to C(Y,j,f^*(L,a))
\]
is a morphism of complexes provided one of the following conditions hold:
\begin{enumerate}
\item $f$ is smooth.
\item $k$ is of characteristic different from $2$.
\end{enumerate}
\end{theorem}

\begin{proof}
If $f$ is smooth, then this is \cite[Lemma 5.16]{Morel08}. If $f$ is flat, then the proof goes back to \cite[Corollaire 10.4.2]{Fasel08a}.
\end{proof}

\begin{rem}
As for push-forwards, the assumption on the base field is irrelevant and should be dropped in the near future.
\end{rem}

As a consequence, we see that the association $X\mapsto \CHW^i(X)$ induces a functor from the category of (essentially) smooth $k$-schemes with flat morphisms to the category of abelian groups. Interestingly, we can also define pull-backs for homological Chow-Witt groups, provided the morphisms are smooth. Let then $f:Y\to X$ be a smooth morphism between schemes which are essentially of finite type. For simplicity, we suppose that the relative dimension $d:=d_Y-d_X$ is constant (for instance $X,Y$ are connected). Let $x\in X_{(i)}$ and let $y\in Y_{(d+i)}$ be such that $f(x)=y$. In that case, we have a commutative diagram of exact sequences
\[
\xymatrix{
0\ar[r] & (\mathfrak m_x/\mathfrak m_x^2)\otimes k(y)\ar[r]\ar[d] & \Omega_{X/k}\otimes k(y)\ar[r]\ar[d] & \Omega_{k(x)/k}\otimes k(y)\ar[r]\ar[d] & 0 \\
0\ar[r] & \mathfrak m_y/\mathfrak m_y^2\ar[r] & \Omega_{Y/k}\otimes k(y)\ar[r] & \Omega_{k(y)/k}\ar[r] & 0}
\]
in which the vertical arrows are induced by $f$. Since $f$ is smooth, the left-hand vertical arrow is an isomorphism and the other two arrows are injective. The cokernel of the middle arrow is $\Omega_{Y/X}\otimes k(y)$ and we obtain an exact sequence
\[
0\to \Omega_{k(x)/k}\otimes k(y)\to \Omega_{k(y)/k}\to \Omega_{Y/X}\otimes k(y)\to 0 
\] 
and a canonical isomorphism $D(\Omega_{k(y)/k})\simeq D(\Omega_{k(x)/k}\otimes k(y))\otimes D(\Omega_{Y/X}\otimes k(y))$. Using the homomorphism $\iota:\KMW_{j+i}(k(x))\to \KMW_{j+i}(k(y))$, we can then define
\[
\KMW_{j+i}(k(x),D(\Omega_{k(x)/k}))\to \KMW_{j+i}(k(y),D(\Omega_{k(y)/k})\otimes D(\Omega_{Y/X}\otimes k(y))^{-1}).
\]
Summing up these homomorphisms, we obtain a homomorphism (of degree $d$)
\[
f^*:\tilde{C}(X,j,L,a)\to \tilde{C}(Y,j-d,D(\Omega_{Y/X})^{-1}\otimes f^*(L,a))
\]
which coincides with the pull-back for cohomological Chow-Witt groups in case $X$ and $Y$ are smooth. We then have the following result (with the usual remark about the characteristic of the base field).

\begin{theorem}
Let $f:Y\to X$ be a smooth morphism of $k$-schemes essentially of finite type over $k$, and let $(L,a)$ be a graded line bundle over $X$. Then, the pull-back homomorphism (of degree $d=d_Y-d_X$)
\[
f^*:\tilde{C}(X,j,L,a)\to \tilde{C}(Y,j-d,D(\Omega_{Y/X})^{-1}\otimes f^*(L,a))
\]
is a morphism of complexes if either of the following two conditions holds:
\begin{enumerate}
\item both schemes are essentially smooth.
\item $k$ is of characteristic different from $2$.
\end{enumerate}
\end{theorem}

\begin{proof}
The first case is an immediate consequence of the previous theorem. For the second case, see \cite[\S 4.2.4]{Deglise17}.
\end{proof}

\begin{theorem}[Homotopy invariance]
Let $X$ be a scheme which is essentially of finite type over $k$ and let $p:X\times \A^1\to X$ be the projection. Suppose that either $k$ is of characteristic different from $2$ or that $X$ is smooth. Then, the homomorphism
\[
p^*:\tilde{C}(X,j,L,a)\to \tilde{C}(X\times \A^1,j-1,D(\Omega_{X\times \A^1/X})^{-1}\otimes p^*(L,a))
\]
is a quasi-isomorphism for any graded line bundle $(L,a)$ over $X$.
\end{theorem}

\begin{proof}
We observe that in the first case the extra assumption is only to make sure that the Rost-Schmid complexes are indeed complexes (which we know without extra conditions if $X$ is smooth).  In any case, the proof is a verbatim of \cite[Satz 6.1.1]{Schmid98} with input the exact sequences of Proposition \ref{prop:twistedhomotopy}.
\end{proof}

As a corollary of this theorem and of the existence of exact sequences of localization discussed in Section \ref{sec:basicprop}, we obtain the following result. 

\begin{cor}\label{cor:homotopyinv}
Let $p:Y\to X$ be an affine bundle over a scheme $X$ essentially of finite type over $k$ and let $r=d_Y-d_X$. Then the pull-back homomorphism
\[
p^*:\tilde{C}(X,j,L,a)\to \tilde{C}(Y,j-r,D(\Omega_{Y/X})^{-1}\otimes p^*(L,a))
\]
is a quasi-isomorphism for any graded line bundle $(L,a)$ over $X$.
\end{cor}

\subsection{Euler classes}

Let $X$ be a scheme (essentially) of finite type over $k$, and let $p:E\to X$ be a vector bundle of rank $r$. Let $s:X\to E$ be the zero section of $E$. For any graded line bundle $(L,a)$ over $X$, we have two morphisms of complexes (the latter being of degree $r$)
\[
s_*:\tilde C(X,j,s^*(D(\Omega_{E/X})^{-1})\otimes (L,a))\to \tilde C(E,j,D(\Omega_{E/X})^{-1}\otimes p^*(L,a))
\]
and 
\[
p^*:\tilde C(X,j+r,L,a)\to \tilde C(E,j,D(\Omega_{E/X})^{-1}\otimes p^*(L,a))
\]
with $\Omega_{E/X}\simeq p^*E^\vee$.
This induces a homomorphism for any $i\in \N$
\[
e(E):\H_i(X,j,D(E)\otimes (L,a))\to \H_{i-r}(X,j+r,L,a)
\]
that we call \emph{the (homological) Euler homomorphism}. Supposing that $X$ is smooth, we may replace $(L,a)$ by the graded line bundle $D(E^\vee)\otimes D(\Omega_{X/k})^{-1}\otimes (L,a)$ to get a homomorphism
\[
e(E):\H_i(X,j,D(\Omega_{X/k})^{-1}\otimes (L,a))\to \H_{i-r}(X,j+r,D(E^\vee)\otimes D(\Omega_{X/k})^{-1}\otimes (L,a))
\]
In view of the identification $\tilde{C}_i(X,j,D(\Omega_{X/k})^{-1})=C(X,j+d_X)^{d_X-i}$, this translates into a homomorphism
\[
e(E):\H^i(X,j,L,a)\to \H^{i+r}(X,j+r,D(E^\vee)\otimes (L,a)).
\]
called \emph{the (cohomological) Euler homomorphism}. Taking $i=j=0$, we can consider $e(E)(\langle 1\rangle)\in \H^{r}(X,r,D(E^\vee))=\CHW^r(X,D(E^\vee))$. This is the \emph{Euler class} of $E$. As a consequence of the projection formula we'll see later, the cohomological Euler homomorphism is the multiplication by the Euler class. 
\section{Products and general pull-backs}

\subsection{Exterior product}

In this section, we show that the cohomological Chow-Witt groups admit exterior products and pull-backs with respect to arbitrary morphisms of (essentially) smooth schemes. Let then $X,Y$ be essentially smooth schemes and let $x\in X^{(i)}$ and $y\in Y^{(i^\prime)}$. The base field being perfect, the product $k(x)\otimes_kk(y)$ is actually a product of fields, say $k(x)\otimes_kk(y)=k(u_1)\times \ldots \times k(u_n)$ for $u_1,\ldots,u_n\in (X\times Y)^{(i+i^\prime)}$. Writing $X_x:=\Spec(\OO_{X,x})$ and $Y_y:=\Spec(\OO_{Y,y})$, we get a Cartesian square
\[
\xymatrix{\Spec (k(x)\otimes k(y))\ar[r]\ar[d] & X_x\times \Spec (k(y))\ar[d] \\
\Spec(k(x))\times Y_y\ar[r] & X_x\times Y_y}
\]
The morphisms in the diagram are all complete intersection, and we then obtain a decomposition (\cite[Appendix B.7.4]{Fulton98})
\[
\left((\mathfrak m_x/\mathfrak m_x^2)^*\otimes k(u)\right)\oplus \left((\mathfrak m_y/\mathfrak m_y^2)^*\otimes k(u)\right)=(\mathfrak m_{u}/\mathfrak m_{u}^2)^*
\]  
for each $u=u_1,\ldots,u_n$. Writing $i_{x,u}:k(x)\to k(u)$ and $i_{y,u}:k(y)\to k(u)$ for the field extensions, we then obtain a canonical isomorphism
\begin{equation}\label{eqn:sumofcones}
i_{x,u}^*D(\mathfrak m_x/\mathfrak m_x^2)^{-1}\otimes i_{y,u}^*D(\mathfrak m_y/\mathfrak m_y^2)^{-1}\simeq D(\mathfrak m_u/\mathfrak m_u^2)^{-1}.
\end{equation}

\begin{defn}
Let $x\in X^{(i)}$, $y\in Y^{(i^\prime)}$ and $u\in (X\times Y)^{(i+i^\prime)}$ be as above. We write $\mu(x,y;u):$
\[
\KMW_*(k(x),D(\mathfrak m_x/\mathfrak m_x^2)^{-1})\times \KMW_*(k(y),D(\mathfrak m_y/\mathfrak m_y^2)^{-1})\to \KMW_*(k(u),D(\mathfrak m_u/\mathfrak m_u^2)^{-1})
\]
for the composite of the product of twisted Milnor-Witt $K$-theory groups and the isomorphism induced by (\ref{eqn:sumofcones}).
\end{defn}

Explicitly, let $\alpha\in \KMW_*(k(x))$, $\overline x_1,\ldots,\overline x_i$ be generators of $\mathfrak m_x/\mathfrak m_x^2$, $\beta\in \KMW_*(k(y))$ and $\overline y_1,\ldots,\overline y_{i^\prime}$ be generators of $\mathfrak m_y/\mathfrak m_y^2$, then
\[
\mu(x,y;u)(\alpha\otimes \overline x_1^*\wedge\ldots\wedge\overline x_i^*,\beta\otimes \overline y_1^*\wedge\ldots\wedge\overline y_{i^\prime}^*)=\alpha\cdot\beta\otimes \overline x_1^*\wedge\ldots\wedge\overline x_i^*\wedge\overline y_1^*\wedge\ldots\wedge\overline y_{i^\prime}^*.
\]

We let the reader prove the following result, which will imply the graded commutativity of the exterior product on Chow-Witt groups.

\begin{lem}\label{lem:gradedcomm}
Let $r,s\in \ZZ$. The following diagram, in which the left-hand map is the switch, 
\[
\xymatrix{\KMW_r(k(x),D(\mathfrak m_x/\mathfrak m_x^2)^{-1})\otimes \KMW_s(k(y),D(\mathfrak m_y/\mathfrak m_y^2)^{-1})\ar[r]^-{\mu(x,y;u)}\ar[d] & \KMW_{r+s}(k(u),D(\mathfrak m_u/\mathfrak m_u^2)^{-1})\ar[d]^-{\epsilon^{rs}\langle(-1)^{ii^\prime}\rangle} \\
\KMW_s(k(y),D(\mathfrak m_y/\mathfrak m_y^2)^{-1})\otimes \KMW_r(k(x),D(\mathfrak m_x/\mathfrak m_x^2)^{-1})\ar[r]_-{\mu(y,x;u)} & \KMW_{r+s}(k(u),D(\mathfrak m_u/\mathfrak m_u^2)^{-1})}
\]
is commutative.
\end{lem}

We may assemble the homomorphisms $\mu(x,y;u)$ to obtain a homomorphism of graded abelian groups
\[
\mu:C(X,j,L,a)\times C(Y,j^\prime,L^\prime,a^\prime)\to C(X\times Y,j+j^\prime,(L,a)\otimes (L^\prime,a^\prime)).
\]
We also often denote $\mu$ by $\times$. Note also the extra subtlety induced by the presence of graded line bundles.  Indeed, in that case, one has to use the canonical isomorphism of graded line bundles over $k(u)$
\[
D(\mathfrak m_x/\mathfrak m_x^2)^{-1}\otimes (L,a)\otimes D(\mathfrak m_y/\mathfrak m_y^2)^{-1}\otimes (L^\prime,a^\prime)\to D(\mathfrak m_x/\mathfrak m_x^2)^{-1}\otimes D(\mathfrak m_y/\mathfrak m_y^2)^{-1}\otimes (L,a)\otimes (L^\prime,a^\prime)
\] 
and proceed as above. This being said, we have the following result (written as for now only in characteristic different from $2$) whose proof can be found in \cite[Corollary 4.12]{Fasel07}.

\begin{prop}
Let $X$ and $Y$ be essentially smooth schemes over $k$. Then, the homomorphism
\[
\mu:C(X,j,L,a)\times C(Y,j^\prime,L^\prime,a^\prime)\to C(X\times Y,j+j^\prime,(L,a)\otimes (L^\prime,a^\prime)).
\]
induces a well-defined product on cohomology provided the base field is of characteristic different from $2$.
\end{prop}

The above proposition is the result of putting together two Leibniz-type formulas (one for Witt groups, the other for $\KM$-cohomology). We refer the reader to \cite{Calmes17b} for the following more precise result (without any extra assumption on the base field).

\begin{prop}[Leibniz formula]
Let $\alpha\in C^i(X,j,L,a)$ and $\beta\in C^{i^\prime}(Y,j^\prime,L^\prime,a^\prime)$. Then
\[
d^{i+i^\prime}_{X\times Y}(\alpha\times\beta)=(d^i_X(\alpha)\times \beta)+(-1)^i\epsilon^j\langle (-1)^a\rangle (\alpha\times d^{i^\prime}_Y(\beta)).
\]
\end{prop}

\subsection{Cartier divisors}

In this section, we explain how to pull-back along a (principal) Cartier divisor, starting with a few definitions and basic facts on the latter following \cite[Appendix B.4]{Fulton98}. Let $X$ be an integral scheme and let $C=\{U_i,f_i\}$ be a Cartier divisor on $X$. Here $U_i\subset X$ is an open subscheme, and the functions $f_i\in k(U_i)=k(X)$ are such that for each $i,j$ we have $f_i/f_j\in \OO_{U_i\cap U_j}^\times$. The support $\vert C\vert$ of $C$ is the closed subset of $X$ for which a local equation of $C$ is not invertible. To $C$, we can associate a line bundle $\OO(C)$ which is defined as the subsheaf of $k(X)$ generated on $U_i$ by $f_i^{-1}$. Recall finally that $C$ is said to be effective if local equations are regular. An effective Cartier divisor always has a canonical section $s_C$, which corresponds to $1\in k(X)$. The vanishing locus of the canonical section is precisely $\vert C\vert$.

Our first goal is to associate to a Cartier divisor an explicit canonical class in the group $\CHW^1_{\vert C\vert}(X,D(\OO(C))^{-1})$. We require here $X$ to be smooth (and integral). Let $x\in X^{(1)}$. There exists $i$ such that $x\in U_i$ and we can consider the residue homomorphism
\[
\partial_x:\KMW_1(k(X),D(\OO(C))^{-1})\to \KMW_0(k(x),D(\mathfrak m_x/\mathfrak m_x^2)^{-1}\otimes D(\OO(C))^{-1}).
\]
We set $\widetilde{\mathrm{ord}}_x(C):=\partial_x([f_i]\otimes f_i)$ and observe that this definition depends a priori on the choice of $i$ such that $x\in U_i$ (we have identified $k(X)$ with $k(U_i)$ on the left-hand side). However, the following computation shows that $\widetilde{\mathrm{ord}}_x(C)$ is actually independent of such a choice. Indeed, suppose that $x\in U_i\cap U_j$. Then
\[
[f_i]\otimes f_i=[(f_i/f_j)\cdot f_j]\otimes (f_i/f_j)\cdot f_j=\left([f_i/f_j]+\langle f_i/f_j\rangle[f_j]\right)\otimes (f_i/f_j)\cdot f_j=[f_i/f_j]\otimes f_i+[f_j]\otimes f_j
\]
and we can use the fact that $f_i/f_j$ is invertible to see that the residue of both $[f_i]\otimes f_i$ and $f_j\otimes f_j$ are the same. This leads to the following definition.

\begin{defn}
We set 
\[
\widetilde{\mathrm{ord}}(C)=\sum_{x\in X^{(1)}\cap \vert C\vert } \widetilde{\mathrm{ord}}_x(C) \in C(X,1,D(\OO(C))^{-1}).
\]
\end{defn}

A simple computation shows that in fact $\widetilde{\mathrm{ord}}(C)$ is a cocycle as could be seen by repeating the proof of \cite[Lemma 2.1.3]{Fasel16b} (and noting that the assumption on the characteristic of $k$ is not used there). If now $\alpha\in C(X,j,L,a)^i$, the question is to know if one can define an element $C\cdot \alpha\in C_{\vert C\vert}(X,j,D(\OO(C))^{-1}\otimes (L,a))^{i+1}$ which coincides with $\widetilde{\mathrm{ord}}(C)$ in case $\alpha=\langle 1\rangle \in C(X,0)^0$. It is possible to do it in fairly general situations, but only for cocycles as explained in \cite{Calmes17b}.

We now pass to the definition of a partial pull-back which will be needed in the sequel. For $f\in\OO(X)$, we denote by $V(f)$ the zero locus of $f$, by $U(f)$ its complement and by $\iota:U(f)\to X$ the inclusion. We sometimes simply write $U$ and $V$ when the context is clear. We can also observe that the codomain of the differential
\[
d^i:C(U(f),j,\iota^*(L,a))^{i}\to C(X,j,L,a)^{i+1}
\]
can be split into $C(X,j,L,a)^{i+1}=C(U(f),j,\iota^*(L,a))^{i+1}\oplus C_{V(f)}(X,j,L,a)^{i+1}$. Consequently, we can write $d^i=d^i_U+d(f)^i$, where $d^i_U$ is the differential in the Rost-Schmid complex of $U$ and $d(f)^i$ is the component supported on $V(f)$. We define a homomorphism
\[
\mu^i_f:C(U(f),j,\iota^*(L,a))^{i}\to C_{V(f)}(X,j+1,D(f)^{-1}\otimes (L,a))^{i+1}
\]
as follows, where $D(f)$ is the Cartier divisor associated to $f$. 

Let $x\in U(f)^{(i)}$ and let $\alpha^\prime=\alpha\otimes (\overline x_1^*\wedge \ldots \wedge \overline x_i^*)\otimes l \in \KMW_{j-i}(k(x),D(\mathfrak m_x/\mathfrak m_x^2)^{-1}\otimes (L,a))$. We set 
\[
\mu^i_f( \alpha^\prime)=d(f)^i(\langle (-1)^i\rangle[f]\alpha\otimes (\overline x_1^*\wedge \ldots \wedge \overline x_i^*)\otimes f\otimes l).
\]
\begin{lem}
The homomorphisms 
\[
\mu^i_f:C(U(f),j,\iota^*(L,a))^{i}\to C_{V(f)}(X,j+1,D(f)^{-1}\otimes (L,a))^{i+1}
\]
induce a morphism of complexes (of degree $1$)
\[
\mu_f:C(U(f),j,\iota^*(L,a))\to C_{V(f)}(X,j+1,D(f)^{-1}\otimes (L,a)).
\]
\end{lem}

\begin{proof}
See \cite[\S 2.3.1]{Asok13}.
\end{proof}

In case $V(f)$ is smooth, we may use the identification of Section \ref{sec:basicprop} to see $\mu_f$ as a homomorphism (of degree $0$)
\[
C(U(f),j,\iota^*(L,a))\to C(V(f),j,(L,a)_{\vert V(f)})
\]
that we still denote by $\mu_f$.

\begin{rem}
Let $f\in \OO(X)$ be a global section as above, and let $\lambda\in \OO(X)^\times$ be an invertible global section. In general, the homomorphisms $\mu_f$ and $\mu_{\lambda f}$ are different, even for cocycles on $U(f)=U(\lambda f)$. However, they coincide for cocycles on $X$ and form a special case of the pull-back along a Cartier divisor hinted at above. To see that $\mu_f$ and $\mu_{\lambda f}$ are in general different, one may consider $X=\A^1$, $f=t$ and $\lambda=-1$. For $\alpha=[t]\in C(\A^1,1)^0$, we find 
\[
\mu_t([t])=[-1]\otimes \overline t^*\otimes t\in C_{V(t)}(\A^1,2,D(t)^{-1})^1
\]
On the other hand, $\mu_{-t}([t])=0$ as $[-t,t]=0$. More generally, we have $\mu_{\lambda t}([t])=\epsilon [-\lambda]\otimes \overline t^*\otimes t$ for $\lambda \in k^\times$.
\end{rem}

\subsection{General pull-backs}

In this section, we construct pull-backs for arbitrary morphisms following \cite{Rost96} and \cite{Fasel07}. We start by briefly recalling the deformation to the normal cone construction (as explained for instance in \cite[\S 10]{Rost96}). This is a fundamental object in the definition of the pull-back for Chow-Witt groups (and more generally a fundamental object for any reasonable cohomology theory). Let $i:Y\subset X$ be an embedding of smooth schemes (say of constant relative dimension $r=d_X-d_Y$) with normal (vector) bundle $\pi:N_YX\to Y$. The deformation to the normal cone is a commutative diagram 
\[
\xymatrix{Y\times \A^1\ar[r]^-{\iota}\ar[rd]_-{i\times 1} & D(Y,X)\ar[d]^-p \\
 & X\times \A^1}
\]
such that:
\begin{enumerate}
\item $D(Y,X)$ is smooth.
\item The restriction of $p$ to $X\times (\A^1\setminus \{0\})$ is an isomorphism.
\item The restriction of $p$ to $X\times \{0\}$ is the composite $N_YX\stackrel{\pi}\to Y\stackrel i\to X$ and $\iota(0):Y\to N_YX$ is the zero section.
\end{enumerate}

If $t$ denotes a coordinate of $\A^1$, we then have a global section $t\in \OO(D(Y,X))$ whose zero locus is $N_YX$ (which is a vector bundle over $Y$). Any graded line bundle $(L,a)$ on $X$ yields a graded line bundle on $D(Y,X)$ whose restriction to $U:=D(Y,X)\setminus N_YX$ is isomorphic via $p$ to the pull-back (via the projection) of $(L,a)$ to $X\times (\A^1\setminus \{0\})$. Besides, the restriction to $N_YX$ of this graded line bundle is $\pi^*(L,a)_{\vert Y}$ (\cite[\S 3.1]{Nenashev07}). We then have morphisms of complexes (of degree $0$)
\[
C(X,j,L,a)\to C(X\times (\A^1\setminus \{0\}),j,L,a)\stackrel{\mu_t}\to C(N_YX,j,\pi^*(L,a)_{\vert Y})
\]
where the first map is the pull-back along the projection $X\times (\A^1\setminus \{0\})\to X$. Let us note that this morphism is very explicit. However, we would like to replace the right-hand term with the Rost-Schmid complex of $Y$. This requires to pass to cohomology and use the homotopy invariance property described in Corollary \ref{cor:homotopyinv}. Unfortunately, it is in general very hard to compute the inverse (in cohomology) of $\pi^*$, making pull-backs usually hard to compute. Note however that in case $N_YX\simeq \A^n_Y$, then there is a canonical homotopy using iterations of the method described in \cite[proof of Theorem 5.38]{Morel08}. In any case, we may now define the pull-back along $i:Y\to X$ as follows.

\begin{defn}
Let $i:Y\to X$ be a closed embedding of smooth schemes. Let further $r,j\in \Z$ and $(L,a)$ be a graded line bundle on $X$. We define
\[
i^*:\H^r(X,j,L,a)\to \H^r(Y,j,(L,a)_{\vert Y})
\]
as $i^*(\alpha)=(\pi^*)^{-1}(\mu_t(\alpha))$.
\end{defn}

This pull-back allows us to define the pull-back along any morphism $f:Y\to X$. Indeed, we can factorize $f$ using its graph
\[
\xymatrix{Y\ar[r]^-{\Gamma_f}\ar[rd]_-f & Y\times X\ar[d]^-p \\
 & X}
\]
and observe that $p$ is smooth and $\Gamma_f$ is an embedding of smooth schemes. 

\begin{defn}
For any $r,j\in \Z$ and $(L,a)\in \mathcal G(X)$, we define
\[
f^*:\H^r(X,j,L,a)\to \H^r(Y,j,f^*(L,a))
\]
as $f^*:=\Gamma^*p^*$ and call $f^*$ the \emph{pull-back along $f$}. 
\end{defn}

\begin{theorem}
The association $f\mapsto f^*$ is well-defined and functorial in case the base field $k$ is of characteristic different from $2$.
\end{theorem}

\begin{proof}
This is \cite[Theorem 2.3]{Asok13}.
\end{proof}

\begin{rem}
As usual, the assumption on $k$ will be dropped in \cite{Calmes17b}.
\end{rem}

\begin{rem}
The reader may have observed that the method above also allows us to define pull-backs along regular embeddings $f:Y\to X$ for which $X$ and $Y$ are not smooth. This induces a (partial) pull-back at the level of the homology of the Rost-Schmid complexes. More precisely, we obtain pull-backs
\[
f^*:\H_i(X,j,L,a)\to \H_{i-r}(Y,j+r,D(N_YX)^{-1}\otimes (L,a))
\]
where $r=d_X-d_Y$. We'll come back to this in \cite{Calmes17b}.
\end{rem}

\subsection{The ring structure}

The general pull-back map introduced in the previous section allows to define a product on the cohomology of a smooth scheme $X$. Indeed, we have an exterior product (well-defined on cohomology groups)
\[
\mu:C(X,j,L,a)\times C(X,j^\prime,L^\prime,a^\prime)\to C(X\times X,j+j^\prime,(L,a)\otimes (L^\prime,a^\prime)).
\]
that we can compose with the pull-back map associated to the diagonal $\Delta:X\to X\times X$. This yields a well-defined product
\[
\H^i(X,j,L,a)\otimes \H^{i^\prime}(X,j^\prime,L^\prime,a^\prime)\to \H^{i+i^\prime}(X,j+j^\prime,(L,a)\otimes (L^\prime,a^\prime))
\]
which is associative (\cite[Proposition 6.6]{Fasel07}). The unit is given by the class of $\langle 1\rangle \in \KMW_0(k)$ (or rather its pull-back to $X$). More precisely, the product defined above yields a structure of a graded $\KMW_0(k)$-algebra. The graded commutativity is essentially given by Lemma \ref{lem:gradedcomm}. We find that the diagram
\[
\xymatrix{\H^i(X,j,L,a)\otimes \H^{i^\prime}(X,j^\prime,L^\prime,a^\prime)\ar[r]\ar[d] & \H^{i+i^\prime}(X,j+j^\prime,(L,a)\otimes (L^\prime,a^\prime))\ar[d]^-{\langle (-1)^{(a+i)(a^\prime+i^\prime)}\rangle \epsilon^{(j-i)(j^\prime-i^\prime)}} \\
\H^{i^\prime}(X,j^\prime,L^\prime,a^\prime)\otimes \H^i(X,j,L,a)\ar[r] & \H^{i+i^\prime}(X,j+j^\prime,(L,a)\otimes (L^\prime,a^\prime))}
\]
is commutative. In particular, the product 
\[
\CHW^i(X,L,a)\otimes \CHW^{i^\prime}(X,L^\prime,a^\prime)\to \CHW^{i+i^\prime}(X,(L,a)\otimes (L^\prime,a^\prime))
\]
is $\langle (-1)^{(a+i)(a^\prime+i^\prime)}\rangle$-commutative. 

\begin{prop}\label{prop:productcomm}
Let $f:X\to Y$ be a morphism of smooth schemes. Then the pull-back map $f^*$ constructed in the previous section is a ring homomorphism.
\end{prop}

\begin{proof}
See \cite[Proposition 7.2]{Fasel07}.
\end{proof}

\subsection{Important properties}

In this section, we finally list the important properties of the (co-)homology of the Rost-Schmid complexes, starting with the homological story. We assume that $k$ is of characteristic different from $2$ for simplicity, with the usual remark about this assumption. We refer to \cite{Calmes17b} for the first result (alternatively, this follows from the smooth base change of \cite[\S 1.1.7]{Deglise17}).

\begin{theorem}[Homological base change]
Let 
\[
\xymatrix{X^\prime\ar[r]^-{f^\prime}\ar[d]_-{p^\prime} & Y^\prime\ar[d]^-p \\
X\ar[r]_-f & Y}
\]
be a Cartesian square of schemes (essentially) of finite type over $k$ and let $d=\mathrm{dim}(Y^\prime)-\mathrm{dim}(Y)$. Suppose that $f$ is proper and that $p$ is smooth. Then, the diagram
\[
\xymatrix@C=3em{\tilde{C}(X^\prime,j-d,D(\Omega_{X^\prime/X})^{-1}\otimes (p^\prime)^*f^*(L,a))\ar[r]^-{(f^\prime)_*\mathrm{can}} & \tilde{C}(Y^\prime,j-d,D(\Omega_{Y^\prime/Y})^{-1}\otimes p^*(L,a)) \\
\tilde{C}(X,j,f^*(L,a))\ar[r]_-{f_*}\ar[u]^-{(p^\prime)^*} & \tilde{C}(Y,j,L,a)\ar[u]_-{p^*}}
\]
is commutative. Here, $\mathrm{can}$ is the isomorphism induced by the canonical isomorphism $\Omega_{X^\prime/X}\simeq (f^{\prime})^*\Omega_{Y^\prime/Y}$ of \cite[Corollary 4.3]{Kunz86}.
\end{theorem}

\begin{theorem}[cdh descent]
Let
\[
\xymatrix{W\ar[r]^-s\ar[d]_-g & V\ar[d]^-f \\
Y\ar[r]_-r & X}
\]
be a Cartesian square of schemes. Assume that $r$ is a closed immersion, $f$ is proper and the induced morphism $(V\setminus W)\to (X\setminus Y)$ is an isomorphism. Then, we have a long exact sequence
\[
\ldots\to \H_i(W,j,(ig)^*(L,a))\xrightarrow{s_*-g_*} \H_i(V,j,f^*(L,a))\oplus \H_i(Y,j,r^*(L,a))\xrightarrow{f_*+r_*} 
\]
\[
\xrightarrow{f_*+r_*} \H_i(X,j,L,a)\stackrel{\delta}\to \H_{i-1}(W,j,(ig)^*(L,a))\to \ldots
\]
\end{theorem}

\begin{proof}
The proof relies on the exact triangle \cite[1.1.12.c]{Deglise17} specialized to the case of the spectrum representing Borel-Moore motivic MW-cohomology (\cite[\S 4.2.2]{Deglise17}). The coniveau spectral sequence associated to this homology theory is preserved by proper push-forwards and one of the lines at page 2 in the spectral sequence coincides with the Rost-Schmid complex  (\cite[\S 4.2.8]{Deglise17}).

Alternatively, one may use the localization sequences associated to the open embeddings $(V\setminus W)\subset V$ and $(X\setminus Y)\subset X$ and prove by hand that the Cartesian square of schemes induces a Cartesian square of Rost-Schmid complexes.
\end{proof}

\begin{theorem}[Nisnevich descent]
Let
\[
\xymatrix{W\ar[r]^-s\ar[d]_-g & V\ar[d]^-f \\
Y\ar[r]_-r & X}
\]
be a Cartesian square of schemes. Assume that $r$ is an open immersion, that $f$ is \'etale and that the induced morphism (of reduced schemes) $(V\setminus W)\to (X\setminus Y)$ is an isomorphism. Then, we have a long exact sequence
\[
\ldots\to \H_i(X,j,L,a)\xrightarrow{r^*+f^*} \H_i(Y,j,r^*(L,a))\oplus \H_i(V,j,f^*(L,a))\xrightarrow{s^*-g^*} 
\]
\[
\xrightarrow{s^*-g^*} \H_i(W,j,(fs)^*(L,a))\stackrel{\delta}\to \H_{i-1}(X,j,L,a)\to \ldots
\]
\end{theorem}

\begin{proof}
This is an easy consequence of the localization sequences of Section \ref{sec:basicprop} associated to the open embeddings $Y\to X$ and $W\to V$, together with the fact that the Rost-Schmid complex of a scheme coincides with the Rost-Schmid complex of its associated reduced scheme.
\end{proof}

Let's now pass to the properties of the cohomology of the Rost-Schmid complex. We first note that both cohomological Nisnevich descent and cohomological cdh descent are special cases of the theorems above, taking only smooth schemes. The only difference with the homological situation is that we have more general pull-backs and that the base change formula is therefore more general. Before stating the result, we start with a few considerations on virtual relative bundles. Let $f:X\to Y$ be a morphism of schemes of finite type. We say that $f$ is an l.c.i. morphism if it admits a factorization
\[
X\stackrel{i}\to Z\stackrel{p}\to Y
\]
such that $p$ is smooth and $i$ is a regular embedding. For instance, any morphism of smooth schemes is l.c.i. (taking the factorization given by the graph). Then, we have two vector bundles on $X$, namely $i^*\Omega_{Z/Y}$ and $C_{X}Z$. We can define the virtual vector bundle $\tau_{pi}:= C_XZ-i^*\Omega_{Z/Y}$ (in the sense of \cite[\S 4.2]{Deligne87}), which is called \emph{the relative bundle (to $f$) with respect to the factorization $f=pi$}. Suppose next that we have a commutative diagram
\[
\xymatrix{ & Z\ar[rd]^-p\ar[dd]^-q & 
\\X\ar[ru]^-i\ar[rd]_-{i^\prime} & & Y \\
 & Z^\prime\ar[ru]_-{p^\prime}}
\]
where $pi=p^\prime i^\prime=f$ and $q$ is \emph{smooth}. In that case, we have exact sequences (e.g. \cite[Appendix B.7.5]{Fulton98})
\[
0\to q^*\Omega_{Z^\prime/Y}\to \Omega_{Z/Y}\to \Omega_{Z/Z^\prime}\to 0
\]
and 
\[
0\to C_XZ^\prime\to C_XZ\to i^*\Omega_{Z/Z^\prime}\to 0.
\]
Therefore, we obtain isomorphisms
\begin{eqnarray*}
C_XZ- i^*\Omega_{Z/Y} & \to & C_XZ^\prime+ i^*\Omega_{Z/Z^\prime}-i^*\Omega_{Z/Z^\prime}- (i^\prime)^*\Omega_{Z^\prime/Y} \\
 & \to & C_XZ^\prime- (i^\prime)^*\Omega_{Z^\prime/Y},
\end{eqnarray*}
i.e. a canonical isomorphism $\tau_{pi}\to \tau_{p^\prime i^\prime}$. If now we only have a diagram
\[
\xymatrix{ & Z\ar[rd]^-p & 
\\X\ar[ru]^-i\ar[rd]_-{i^\prime} & & Y \\
 & Z^\prime\ar[ru]_-{p^\prime}}
\]
with $pi=p^\prime i^\prime=f$, then we can consider the factorization
\[
X\xrightarrow{i\times i^\prime} Z\times_Y Z^\prime\xrightarrow{\pi} Y
\]
where $\pi$ is the obvious projection from the fiber product to $Y$. It is again a factorization of $f$, and the respective projections to $Z$ and $Z^\prime$ induce morphisms of factorizations as above. We then get a diagram
\[
\xymatrix{ & \tau_{pi} \\
\tau_{\pi(i\times i^\prime)}\ar[ru]\ar[rd] & \\
 & \tau_{p^\prime i^\prime}}
\]
in the category of virtual vector bundles (or the same in the category of graded line bundles). This construction allows to define $\tau_f$ as the limit (over all factorizations) of $\tau_{pi}$ and obtain a well-defined $D(\tau_f)\in G(X)$.

In case both schemes $X$ and $Y$ are essentially smooth, we can use the canonical isomorphisms (\ref{eqn:canonical1})
\[
p^*D(\Omega_{Y/k})\simeq D(\Omega_{Z/k})\otimes D(\Omega_{Z/Y})^{-1}
\]
and (\ref{eqn:canonical3})
\[
D(C_XZ)\otimes D(\Omega_{X/k})\simeq i^*D(\Omega_{Z/k})
\]
for any factorization $f=pi$ to find a canonical isomorphism
\[
D(\tau_{pi})= D(C_XZ)\otimes i^*D(\Omega_{Z/Y})^{-1}\simeq i^*D(\Omega_{Z/k})\otimes D(\Omega_{X/k})^{-1}\otimes i^*D(\Omega_{Z/k})^{-1}\otimes f^*D(\Omega_{Y/k})
\]
Switching the two middle factors and cancelling, we find an isomorphism 
\[
\alpha(X,Y):D(\tau_{pi})\simeq D(\Omega_{X/k})^{-1}\otimes f^*D(\Omega_{Y/k}).
\]
We let the reader check that this induces an isomorphism $D(\tau_f)\simeq D(\Omega_{X/k})^{-1}\otimes f^*D(\Omega_{Y/k})$ that we still denote by $\alpha(X,Y)$.

Recall next that a Cartesian square of schemes (of finite type)
\[
\xymatrix{X^\prime\ar[r]^-{f^\prime}\ar[d]_-{g^\prime} & Y^\prime\ar[d]^-g \\
X\ar[r]_-f & Y}
\]
is \emph{tor-independent} if for any $x\in X$ and $y^\prime\in Y$ with $y=f(x)=g(y^\prime)$ we have 
\[
\mathrm{Tor}_i^{\OO_{Y,y}}(\OO_{X,x},\OO_{Y^\prime,y^\prime})=0
\]
for any $i>0$. If $g$ is l.c.i., we can form the following diagram in which all squares are Cartesian
\[
\xymatrix{X^\prime\ar[r]^-{f^\prime}\ar[d]_-{i^\prime} & Y^\prime\ar[d]^-i \\
X^{\prime\prime}\ar[r]^-{f^{\prime\prime}}\ar[d]_-{p^\prime} &  Y^{\prime\prime}\ar[d]^-p\\
X\ar[r]_-f & Y}
\]
and $p,p^\prime$ are smooth, while $i$ is a regular embedding. 

\begin{lem}
The Cartesian square 
\[
\xymatrix{X^\prime\ar[r]^-{f^\prime}\ar[d]_-{i^\prime} & Y^\prime\ar[d]^-i \\
X^{\prime\prime}\ar[r]_-{f^{\prime\prime}} &  Y^{\prime\prime}}
\]
is tor-independent. Consequently, $i^\prime$ is a regular embedding and the homomorphism
\[
(f^\prime)^*C_{Y^\prime}Y^{\prime\prime}\to C_{X^\prime}X^{\prime\prime}
\]
is an isomorphism.
\end{lem}

\begin{proof}
The property of being tor-independent being local, we may restrict to the affine case and moreover suppose that $i$ is complete intersection. We have a commutative diagram of rings
\[
\xymatrix{B^\prime & A^\prime\ar[l] \\
B^{\prime\prime}\ar[u] & A^{\prime\prime}\ar[l]\ar[u] \\
B\ar[u] & A\ar[l]\ar[u]}
\]
in which the squares are cocartesian, $A^{\prime\prime}$ (resp. $B^{\prime\prime}$) is a smooth $A$-algebra (resp. smooth $B$-algebra) and $A^\prime=A^{\prime\prime}/I$ where $I$ is a complete intersection ideal. Let $P_{\bullet}\to B$ be a projective resolution of $B$ as an $A$-module. We have a commutative square of complexes of $A^{\prime}$-modules
\[
\xymatrix{((P_{\bullet})\otimes_A A^{\prime\prime})\otimes_{A^{\prime\prime}} A^\prime\ar[r]\ar[d] & (B\otimes_A A^{\prime\prime})\otimes_{A^{\prime\prime}} A^\prime\ar[d] \\
(P_{\bullet})\otimes_A A^{\prime}\ar[r] & B\otimes_A A^{\prime}}
\]
in which the vertical arrows are isomorphisms. Since $\mathrm{Tor}_i^{A}(B,A^\prime)=0$ for $i>0$, the bottom arrow is a quasi-isomorphism and it follows that the top arrow is also a quasi-isomorphism. Since $A^{\prime\prime}$ is in particular flat over $A$, $(P_{\bullet})\otimes_A A^{\prime\prime}\to B\otimes_A A^{\prime\prime}=B^{\prime\prime}$ is a projective resolution and we conclude that $\mathrm{Tor}_i^{A^{\prime\prime}}(B^{\prime\prime},A^\prime)=0$ for $i>0$, i.e. that the top square is tor-independent. Since $A^\prime=A^{\prime\prime}/I$ is defined by a regular sequence, we see that the Koszul complex $K$ associated to $I$ is a free resolution of $A^{\prime}$ as $A^{\prime\prime}$-module. Since $\mathrm{Tor}_i^{A^{\prime\prime}}(B^{\prime\prime},A^\prime)=0$ for $i>0$, we see that $K\otimes_{A^{\prime\prime}} B^{\prime\prime}\to A^\prime\otimes_{A^{\prime\prime}}B^{\prime\prime}=B^\prime$ is a quasi-isomorphism. It follows that $B^\prime$ is also defined by a regular sequence and that 
\[
B^{\prime\prime}\otimes_{A^{\prime\prime}} (I/I^2)\simeq I{B^{\prime\prime}}/(I{B^{\prime\prime}})^2.
\]
\end{proof}

As a result, a tor-independent Cartesian square
\[
\xymatrix{X^\prime\ar[r]^-{f^\prime}\ar[d]_-{g^\prime} & Y^\prime\ar[d]^-g \\
X\ar[r]_-f & Y}
\]
in which $g$ is l.c.i. comes equipped with a morphism of virtual vector bundles $(f^{\prime})^*(\tau_g)\to \tau_{g^\prime}$, which is also independent of the choice of the factorization of $g$. This yields a well-defined isomorphism
\[
\mathrm{can}:(f^\prime)^*D(\tau_g)\simeq D(\tau_{g^\prime})
\]
which is the one needed in the base change formula that we can finally state.

\begin{theorem}[Cohomological base change]
Let 
\[
\xymatrix{X^\prime\ar[r]^-{f^\prime}\ar[d]_-{g^\prime} & Y^\prime\ar[d]^-g \\
X\ar[r]_-f & Y}
\]
be a tor-independent Cartesian square of smooth schemes over a field of characteristic different from $2$. Suppose that $f$ is proper of relative dimension $d=\mathrm{dim}(Y)-\mathrm{dim}(X)$. Then, the following diagram commutes
\[
\xymatrix{\H^{i}(X^\prime,j,D(\tau_{g^\prime})^{-1}\otimes D(\Omega_{X^\prime/k})\otimes (fg^\prime)^*(L,a))\ar[r]^-{f^\prime_*\mathrm{can}} & \H^{i+d}(Y^\prime,j+d,D(\tau_{g})^{-1}\otimes D(\Omega_{Y^\prime/k})\otimes g^*(L,a))  \\
\H^{i}(X^\prime,j,(g^\prime)^*D(\Omega_{X/k})\otimes (fg^\prime)^*(L,a))\ar[u]^-{\alpha(X^\prime,X)} & \H^{i+d}(Y^\prime,j,g^*D(\Omega_{Y/k})\otimes g^*(L,a))\ar[u]_-{\alpha(Y^\prime,Y)}  \\
\H^{i}(X,j,D(\Omega_{X/k})\otimes f^*(L,a))\ar[r]_-{f_*}\ar[u]^-{(g^\prime)^*} &  \H^{i+d}(Y,j+d,D(\Omega_{Y/k})\otimes (L,a))\ar[u]_-{g^*} }
\]
for any $i,j\in\ZZ$ and any graded line bundle $(L,a)$ over $Y$.
\end{theorem}

We refer the reader to \cite{Calmes17b} for a (characteristic free) proof. In fact, the result holds more generally if we assume that the schemes are of finite type over a field, and that the vertical map are l.c.i.

One of the main consequences of the base change theorem is a projection formula for the cohomology of the Rost-Schmid complex. 

\begin{theorem}[Projection formula]
Let $f:X\to Y$ be a proper morphism of essentially smooth schemes over $k$. Let $d=\mathrm{dim}(Y)-\mathrm{dim}(Y)$ and let $(L,a)$ and $(L^\prime,a^\prime)$ be graded line bundles over $Y$. For any $\alpha\in \H^i(Y,j,L,a)$ and $\beta\in \H^{i^\prime}(X,j^\prime,D(\Omega_{X/k})\otimes f^*(L^\prime,a^\prime))$ we have 
\[
f_*(f^*(\alpha)\cdot \beta)=\alpha\cdot f_*(\beta) \in \H^{i+i^\prime+d}(Y,j+j^\prime+d,D(\Omega_{Y/k})\otimes (L,a)\otimes (L^\prime,a^\prime))
\]
and
\[
f_*(\beta\cdot f^*(\alpha))=f_*(\beta)\cdot \alpha \in \H^{i+i^\prime+d}(Y,j+j^\prime+d,D(\Omega_{Y/k})\otimes (L^\prime,a^\prime)\otimes (L,a)).
\]
\end{theorem}

\begin{proof}
We can consider the Cartesian square
\[
\xymatrix{X\ar[r]^-{\Gamma_f}\ar[d]_-f & X\times Y\ar[d]^-{f\times \mathrm{id}} \\
Y\ar[r]_-{\Delta_Y} & Y\times Y}
\]
where $\Delta_Y$ is the diagonal morphism and $\Gamma_f$ is the graph of $f$. If $\Delta_X$ is the diagonal morphism of $X$, we have by definition $\Gamma_f=(\mathrm{id}\times f)\circ \Delta_X$. By base change, we have (using the relevant canonical isomorphisms)
\[
\Delta_Y^*(f\times\mathrm{id})_*=f_*\Gamma_f^*=f_*\Delta_X^*(\mathrm{id}\times f)^*.
\]
We let the reader check that $(f\times\mathrm{id})_*=f_*\times \mathrm{id}$, which gives immediately the second formula. The same argument switching $X$ and $Y$ on the right yields the other statement (see also \cite[Remark 3.6]{Calmes14b} for an alternative argument).
\end{proof}

\section{Some computations}

The theory of Chow-Witt groups and more generally of the (co-)homology of the Rost-Schmid complex being relatively new, there are few computations available in the literature, especially compared to the amount of computations involving Chow groups. One primary difficulty is already to compute the Chow-Witt group of the base field, i.e. its Grothendieck-Witt group $\GW(k)$. For a general field, this is a very interesting but quite difficult question, that we completely avoid here. So, all the computations below will be in term of $\GW(k)$, not bothering about its actual shape.

Perhaps the most basic computation of the Chow ring is the one of the projective space $\PP_k^n$. It is a truncated polynomial algebra generated by the first Chern class of the tautological line bundle. If one can define the Euler class of the tautological bundle, it doesn't live in the regular Chow-Witt group, but rather on its version twisted by the graded line bundle $D(\OO(-1))$. This apparently innocuous difference leads to a completely different result as the ordinary one (\cite[Corollary 11.8]{Fasel09d}). 

\begin{theorem}
We have 
\begin{eqnarray*}
\CHW^i(\PP^n_k) & = & \begin{cases}
\GW(k) & \text{if $i=0$ or $i=n$ with $n$ odd.} \\
\Z & \text{if $i$ is even and $i\neq 0$.}  \\
2\Z & \text{if $i$ is odd and $i\neq n$.}
\end{cases} \\
\CHW^i(\PP^n_k,D({\OO(-1)})) & = & \begin{cases}
2\Z & \text{if $i$ is even and $i\neq n$}. \\
\Z & \text{if $i$ is odd}. \\
\GW(k) & \text{if $i=n$ and $n$ is even}.
\end{cases} 
\end{eqnarray*}
\end{theorem} 

In the statement, both the factors $\Z$ and $2\Z$ indicate that the relevant group is $\Z$, but that its image under the projection $\CHW^i(\PP^n_k)\to \CH^i(\PP^n_k)=\Z$ is either everything or of index $2$. More generally, one can compute the whole cohomology of the Rost-Schmid complex of $\PP^n_k$ (\cite[Theorem 11.7]{Fasel09d}). As an interesting byproduct, we can compute the Chow-Witt group (or actually also the ring structure) of the classifying space $\mathrm{B}\mathbb{G}_m$. To define the latter, consider for $n\in\N$ the embedding $k^{n+1}\to k^{n+2}$ as the first $n+1$ factors. It induces a closed immersion $b_n:\PP^{n}\to \PP^{n+1}$ and we obtain a sequence
\[
\Spec k\stackrel{b_0}\to \PP^1\to \ldots \to \PP^n\stackrel{b_n}\to \PP^{n+1}\to \ldots
\]
For any $i\in\N$, this yields a sequence
\[
\xymatrix@C=1.3em{\CHW^i(\Spec k) & \CHW^i(\PP^1)\ar[l]_-{b_1^*} & \ldots\ar[l] & \CHW^i(\PP^n)\ar[l] & \CHW^i(\PP^{n+1})\ar[l]_-{b_n^*} & \ldots\ar[l]}
\] 
This sequence eventually stabilizes (\cite[Theorem 3.3]{Asok13}), giving a "concrete" model for the limit, which is by definition the group $\CHW^i(\mathrm{B}\mathbb{G}_m)$. In short, we have $\CHW^i(\mathrm{B}\mathbb{G}_m)=\CHW^i(\PP^n)$ for $n$ large enough. We can define the group $\CHW^i(\mathrm{B}\mathbb{G}_m,D(\OO(-1)))$ in a similar fashion (\cite[Theorem 3.3]{Asok13}). As a straightforward consequence of the previous theorem, we obtain the following result.

\begin{theorem}
We have 
\[
\CHW^i((\mathrm{B}\mathbb{G}_m)=\begin{cases} \GW(k) & \text{ if $i=0$.} \\ \Z & \text{ if $i\neq 0$ is even.} \\ 2\Z & \text{ if $i$ is odd.} \end{cases}
\]
while 
\[
\CHW^i((\mathrm{B}\mathbb{G}_m,D(\OO(-1)))=\begin{cases} \Z & \text{ if $i$ is odd.} \\ 2\Z & \text{ if $i$ is even.} \end{cases}
\]
\end{theorem}

In terms of characteristic classes, the total ring
\[
\CHW^{\mathrm{tot}}(\mathrm{B}\mathbb{G}_m):=\CHW^*(\mathrm{B}\mathbb{G}_m)\oplus \CHW^*((\mathrm{B}\mathbb{G}_m,D(\OO(-1)))
\]
is generated in degree $1$ by the class $e:=e(\OO(1))\in \CHW^1((\mathrm{B}\mathbb{G}_m,D(\OO(-1)))$ and by another class that we now describe. Recall from Section \ref{sec:preliminaries} that we have a homomorphism
\[
h_n:\KM_n(F)\to \KMW_n(F)
\]
defined by $\{a_1,\ldots,a_n\}\mapsto [a_1^2,a_2,\ldots,a_n]$. This induces a homomorphism
\[
H_i:\CH^i(X)\to \CHW^i(X)
\]
for any smooth scheme $X$. In particular, we can consider for any $i\in\N$ the class 
\[
c:=H_1(c_1(\OO(1)))\in \CHW^1(\mathrm{B}\mathbb{G}_m).
\]
It turns out that 
\[
\CHW^{\mathrm{tot}}(\mathrm{B}\mathbb{G}_m)=\GW(k)[e,c]/J
\]
where $J$ is the ideal generated by the relations:
\begin{enumerate}
\item $\I(k)e=0=\I(k)c$.
\item $ec=ce$.
\item $c^2=2he^2$.
\end{enumerate}

The interest of computing the total Chow-Witt ring of $\mathrm{B}\mathbb{G}_m$ lies in the fact that the latter gives a complete set of characteristic classes attached to line bundles (\cite[Theorem 1.3]{Totaro99} and \cite[Proposition 3.3]{Hornbostel18}). Thus, the outcome of our computation is that there are essentially two characteristic classes attached to a line bundle: the Euler class (which lives in a twisted group) and a class which is the image of the first Chern class.

Actually, the situation we just described can be generalized to more general (reasonable) group schemes $G$ over $k$. Hornbostel-Wendt computed recently the total Chow-Witt ring of both $\mathrm{BSL}_n$ and $\mathrm{BSp}_{2n}$  (no twist by a line bundle is needed here). We refer the reader to their paper (\cite{Hornbostel18}) for a precise formulation, the essential fact being that there is essentially one new kind of characteristic class: the Borel classes of Panin-Walter (\cite{Panin10pred}); the other classes that appear are roughly defined through Chow groups and Chow groups modulo $2$. Their computation was later extended to $\mathrm{BGL}_n$ by Wendt in \cite{Wendt18}. The reader can compare with our computation above in case $n=1$. The obvious question is now to understand the classifying spaces of orthogonal groups. It is not clear to me what can be expected.

\bibliographystyle{amsplain}
\bibliography{General}

\end{document}